\documentclass[12pt]{article}
\usepackage[T1]{fontenc}
\usepackage{lmodern}
\usepackage[latin1]{inputenc}
\usepackage{latexsym}                 
\usepackage{color}                 
\usepackage{amssymb}
\usepackage{amsmath}
\usepackage{amsthm}
\usepackage{graphicx}
\usepackage{hyperref}
\usepackage{verbatim}
\usepackage{mathptmx}      
\newcommand{\eq}{\end{equation}}
\def\fa{\hbox{ for all }}
\def\span{\hbox{ span }}

\def\dfrac#1#2{\displaystyle{\frac{#1}{#2}   }}

\def\b1{\mathbf 1}

\newcommand{\R}{\ensuremath{\mathbb{R}}}

\newcommand{\N}{\ensuremath{\mathbb{N}}}

\def\calh{{\cal H}}

\def\biglf{\par\bigskip\noindent}
\newcommand{\red}[1]{{\color{red} #1 }}     
\newtheorem{Lemma}{Lemma}

\def\eref#1{(\ref{#1}%
)}
\def\RSref#1{\ref{#1}%
}
\def\RSlabel#1{\label{#1}%
}
\def\RScite#1{\cite{#1}%
}

\newcommand{\bql}[1]{%
\begin{equation}\label{#1}%
}

\def\filename#1{}

\def\dist{\mathrm{dist}}
\def\biglf{\par\bigskip\noindent}
\begin{document}
\begin{center}
  {\large \bf A Greedy Method for Solving Classes of PDE Problems}
\biglf 
Robert Schaback \\
\biglf
Draft of \today
\end{center}
Motivated by the successful use of greedy
algorithms for Reduced Basis Methods,
a greedy method
is proposed that selects $N$ input data in an asymptotically
optimal way to solve well-posed operator equations
using these $N$ data. The operator equations are defined
as infinitely many equations given via a compact set of functionals
in the dual of an underlying Hilbert space, and then the
greedy algorithm, defined directly in the dual Hilbert space,
selects $N$ functionals step by step. When $N$ functionals are
selected, the operator equation
is numerically solved by
projection onto the span of the Riesz representers of the functionals.
Orthonormalizing these yields useful Reduced Basis functions.
By recent results on greedy methods in Hilbert spaces,
the convergence rate
is asymptotically given by Kolmogoroff $N$-widths and
therefore optimal in that sense. However, these $N$-widths
seem to be unknown in PDE applications. Numerical experiments show
that for solving elliptic second-order Dirichlet problems,
the greedy method of this paper behaves like the known $P$-greedy
method for interpolation, applied to second derivatives. 
Since the latter technique is known to realize Kolmogoroff
$N$-widths for interpolation,  it is hypothesized
that the Kolmogoroff $N$-widths for solving second-order PDEs
behave like the  Kolmogoroff $N$-widths for second derivatives,
but this is an open theoretical problem.
\biglf
\biglf
\section{Introduction}\RSlabel{SecIntro}
For illustration of the application scope of this paper,
consider the class of all second-order
elliptic boundary value problems
\bql{eqLufg}
\begin{array}{rcl}
  Lu&=&f\;\hbox{ on } \overline{\Omega}\subset\R^d \\
  u&=& g\;\hbox{ on } \Gamma:=\partial\Omega
\end{array} 
\eq
with arbitrary Dirichlet data
and a fixed second-order strongly elliptic operator $L$ 
on a fixed bounded Lipschitz domain $\Omega\subset\R^d$.
We keep this problem class in strong form and pose it in Sobolev space 
$W_2^m(\R^d)$ with $m>2+d/2$ and spaces of data functions $f$ and $g$
of corresponding smoothness. Similar to Reduced Basis
(e.g. \RScite{maday-et-al:2002-1,binev-et-al:2011-1,%
buffa-et-al:2012-1,devore-et-al:2013-1}) or
Proper Orthogonal Decomposition methods
(e.g. \RScite{moore:1979-1,kunisch-volkwein:1999-1,gubisch-volkwein:2016-1}),
we focus on a {\em class}
of PDE problems, not on single problems. The output of this paper will
be connected to both areas, since a ``reduced''
orthonormal basis is produced
that is adapted to the given class of PDE problems.
\biglf
The next section will generalize such problems to operator equations defined by
sets $\Lambda$
of infinitely many functionals on an underlying Hilbert space of functions,
e.g. a Sobolev space.
In this context, well-posedness can be formulated, and for $N$ selected
functionals in a set $\Lambda_N\subset\Lambda$,
numerical solutions can be obtained by Hilbert space projection on
the Riesz representers of these functionals. This is the well-known
Rayleigh-Ritz idea. For kernel-based spaces, it coincides with
Symmetric Collocation and yields the optimal recovery technique
in Hilbert space for the given data functionals.
\biglf
Section \RSref{SecEAoPM} analyzes the error in terms of the
{\em Generalized Power Function}
$$
P_{\Lambda_N}(\lambda):=\dist(\lambda,\span(\Lambda_N)) 
$$
on the dual of the Hilbert space 
and introduces 
$$
\sigma_{\Lambda_N}:=\displaystyle{\max_{\lambda\in\Lambda}P_{\Lambda_N}(\lambda)   }\\ 
$$
that controls the error of the numerical solution in Hilbert space.
\biglf
The greedy method of Section \RSref{SecGM} now selects
$$
\lambda_{N+1}:=\hbox{ arg max} P_{\Lambda_N}(\lambda)
$$
and follows the results of the literature on Reduced Basis Methods,
linking the decay of $\sigma_{\Lambda_N}$ to Kolmogoroff $N$-widths.
\biglf
For operator equations defined on Hilbert spaces of functions
on a bounded domain $\Omega$, the additional
quantity
$$
\rho_{\Lambda_N}:=\displaystyle{\max_{x\in\overline\Omega}P_{\Lambda_N}(\delta_x)   }
$$
directly controls the pointwise and uniform error in the domain,
but is not useful for greedy methods. For well-posed problems,
Section \RSref{SecEA} will show that $\rho_{\Lambda_N}$ is bounded
form above by $\sigma_{\Lambda_N}$ up to a factor, and
numerical results in the final Section \RSref{SecNumRes}
suggest that this bound is asymptotically
sharp.   
\biglf
Before that, Section \RSref{SecSobCas} gives a partial analysis of
expectable Kolmogoroff $N$-widths for second-order Dirichlet
problems in Sobolev spaces $W_2^m$. It is hypothesized that
both $\sigma_{\Lambda_N}$ and the $N$-width behave like
${\cal O}(N^{-\frac{m-2-d/2}{d}})$ for $N\to \infty$, which is the
Kolmogoroff $N$-width for interpolation problems
in  $W_2^{m-2}$ with respect to the supremum norm.
Section \RSref{SecNumRes} shows supporting examples
and concludes the paper.
\section{Hilbert Space Theory}\RSlabel{SecHSTheory}
Following \RScite{hon-et-al:2003-1,schaback:2016-4}
the problem class is written in terms of
infinitely many constraints, each defined by a linear functional.
The functional sets
then are
\bql{eqLamdef}
\begin{array}{rcl}
\Lambda_1&:=&\{\delta_x\circ L\;:\;x\in \overline{\Omega}\}\\
\Lambda_2&:=&\{\delta_x\;:\;x\in \Gamma:=\partial {\Omega}\}
\end{array} 
\eq
combined into $\Lambda:=\Lambda_1\cup\Lambda_2$.
Since all single functionals are continuous
on $W_2^m(\R^d)$, the above sets are
images of compact sets by continuous maps, thus compact.
This brings us into line with the literature on reduced basis methods. 
\biglf
The above problems are {\em well-posed}
in the sense that there is a standard
{\em well-posedness inequality} of the form
\cite[1.5, p. 30]{braess:2001-1}
\bql{eqWP}
\|u\|_{\infty,\overline{\Omega}}\leq \|u\|_{\infty,\partial{\Omega}}
+C\|Lu\|_{\infty,\overline{\Omega}}\leq
(C+1)\sup_{\lambda\in\Lambda}|\lambda(u)|
\fa u\in C^2(\overline{\Omega})\cup C(\Gamma).
\eq
\subsection{Abstract Problem}\RSlabel{SecAP}
Generalizing this case,
and following \RScite{hon-et-al:2003-1,schaback:2016-4}, we
assume a Hilbert space $\calh$ and a subset $\Lambda$ of its dual $\calh^*$
that is {\em total} in the sense that
$$
\lambda(u)=0\fa \lambda \in \Lambda \hbox{ implies } u=0. 
$$
If we formally introduce the linear {\em data map}
$D_\Lambda\;:\calh\to \R^\Lambda$  with
$$
D_\Lambda(u):=\{\lambda(u)\}_{\lambda\in\Lambda} \fa u\
\in \calh,
$$
this means that elements $u\in \calh$  are uniquely identifiable
from their {\em data} $D_\Lambda(u)$. The central background
problem in this paper
is to recover elements $u$ from their data $D_\Lambda(u)$ in practice,
i.e. the approximate numerical inversion of the data map.
In view of the preceding example we assume that the set $\Lambda$ is compact.  
\biglf
The invertibility of the data map is quantified by assuming
a {\em well-posedness inequality}
\bql{equWP}
\|u\|_{WP}\leq C_{WP}\|D_\Lambda(u)\|_\infty=C_{WP}\sup_{\lambda\in
  \Lambda}|\lambda(u)|
\fa u\in \calh
\eq
in some {\em well-posedness} norm $\|.\|_{WP}$ on $\calh$
that usually is weaker than the  norm on $\calh$.
Recall that \RScite{schaback:2016-4} allows also to handle weakly formulated
problems as well this way. Furthermore, the framework applies to general
operator equations, including the case of interpolation if the operator
is the identity.
\section{Error Analysis of Projection Methods}\RSlabel{SecEAoPM}
For a finite subset $\Lambda_N:=\{\lambda_1,\ldots,\lambda_N\}$
of $\Lambda$ we can define the subspace
$$
L_N:=\span\{\Lambda_N\}\subseteq \calh^*
$$
and use the Riesz representers $v_{\lambda_1},\ldots, v_{\lambda_N}$
of $\lambda_1,\ldots,\lambda_N$ as trial functions. They span a space
$V_{\Lambda_N}$
that
will lead later to reduced bases.
\biglf
The standard optimal recovery of an element $u\in\calh$ from finite data
$\lambda_1(u),\ldots,\lambda_N(u)$ then proceeds by Hilbert space projection,
i.e. by solving the linear system
$$
\lambda_k(u)=\displaystyle{\sum_{j=1}^N\alpha_j\lambda_k(v_{\lambda_j})
  =\sum_{j=1}^N\alpha_j(v_{\lambda_k},v_{\lambda_j})_\calh
  =\sum_{j=1}^N\alpha_j(\lambda_k,\lambda_j)_{\calh^*}
},\;1\leq k\leq N 
$$
to get the numerical approximation
$$
\tilde u_N:=\sum_{j=1}^N\alpha_jv_{\lambda_j}.
$$
This satisfies the orthogonality relation
\bql{eqorthou}
\|u\|_\calh^2=\|u-\tilde u_N\|_\calh^2+\|\tilde u_N\|_\calh^2
\eq
that implies uniform stability in Hilbert space.
In case of the example in the beginning,
the method is known as {\em Symmetric Collocation}. This is a numerical
technique \RScite{fasshauer:1997-1}
based on \RScite{wu:1992-1} with certain
optimality properties \RScite{schaback:2015-3}
and a convergence theory
\RScite{franke-schaback:1998-1,franke-schaback:1998-2a}. Within the Hilbert
space framework, it produces pointwise optimal approximations to the
true solution under all possible methods that use the same data
\RScite{schaback:2015-3}.

\subsection{Power Function}\RSlabel{SecGPF}
The standard error analysis in Hilbert Spaces uses
the {\em generalized Power Function} defined as  
\bql{eqPFdef} 
P_{\Lambda_N}(\mu):=\displaystyle{
  \min_{\lambda\in L_N}\|\mu-\lambda\|_{\calh^*}   }=:\dist
(\mu,L_N)_{\calh^*} \fa \mu \in \calh^*.
\eq
This is continuous and attains its maximum on the compact set $\Lambda$.
In particular, we are interested in
\bql{eqsigmadef}
\sigma_{\Lambda_N}
:=\sup_{\lambda\in \Lambda}P_{\Lambda_N}(\lambda)=\sup_{\lambda\in \Lambda}\dist
(\lambda,L_N)_{\calh^*}.
\eq
\subsection{Error Analysis}\RSlabel{SecEA}
This quantity leads to error bounds for the discretized
recovery problem that we described in the beginning of this section.
\begin{Lemma}\RSlabel{Lemsigma}
  Let $u\in \calh$ supply the finite data $\lambda_1(u),\ldots,\lambda_N(u)$
  that is used to construct $\tilde u_N$ by projection, and assume
  well-posedness
  in the sense of \eref{equWP}. Then
$$
\begin{array}{rcl}
  \|u-\tilde u_N\|_{WP}
  &\leq &C_{WP}\sigma_{\Lambda_N}\|u\|_\calh.
\end{array}
$$
\end{Lemma} 
\begin{proof}
Since $u$ and $\tilde u_N$ share the same data, the assertion follows from
\bql{eqstandEB}
\begin{array}{rcl}
  \|u-\tilde u_N\|_{WP}
  &\leq &C_{WP}\sup_{\lambda\in \Lambda}|\lambda(u)-\lambda(\tilde u_N)|\\
  &= &C_{WP}\sup_{\lambda\in \Lambda}|(\lambda-\mu)(u-\tilde u_N)|\fa \mu\in L_N
\end{array}
\eq
using orthogonality \eref{eqorthou} and well-posedness \eref{equWP}.
\end{proof}
\noindent Therefore we are interested to find sets $\Lambda_N$ that minimize
$\sigma_{\Lambda_N}$ under all sets of $N$ functionals from $\Lambda$.
\biglf
By the same argument, for all continuous
test functionals $\mu$ the inequality
\bql{eqpwmu}
 |\mu(u)-\mu(\tilde u_N)|\leq P_{\Lambda_N}(\mu)\|u\|_\calh,
\eq
 holds and we can check derivative errors once we evaluate
 the Generalized Power Function on derivative functionals.
 Taking all functionals of $\Lambda$ here, we get
 $$
 \|D_\Lambda(u-\tilde u_N)\|_\infty
 =\displaystyle{\sup_{\lambda\in \Lambda}|\lambda(u-\tilde u_N)|
 \leq \sigma_{\Lambda_N}} \|u\|_\calh,
 $$
 i.e. $\sigma_{\Lambda_N}$ bounds the error in the data norm.
\biglf
In case of a Hilbert space $\calh$ containing
continuous functions on some compact set
$\overline\Omega$, there is a special case of
\eref{eqpwmu} in the pointwise form
\bql{eqpwerrbnd}
\begin{array}{rcl}
  |u(x)-\tilde u_N(x)|
 &\leq & P_{\Lambda_N}(\delta_x)\|u\|_\calh
\end{array}
\eq
that may be checked by evaluation of the Generalized Power Function
over all delta functionals for all points of the domain. If we define
\bql{eqrhodef}
\rho_{\Lambda_N}:=\displaystyle{\max_{x\in\overline\Omega}P_{\Lambda_N}(\delta_x)   } 
\eq
we get a uniform bound
\bql{eqrhoEB} 
\|u-\tilde u_N\|_{\infty,\overline\Omega}\leq \rho_{\Lambda_N}\|u\|_\calh
\eq
that is numerically available, and similar to the  well-posedness inequality
\eref{eqWP}.
\biglf
By definition, the Generalized
Power Function decreases at all functionals when we
extend the set $\Lambda_N$. All of these error bounds will then improve.
For
fixed
$\Lambda_N$ and $\calh$, they are optimal, but here we are interested in
finding good finite subsets $\Lambda_N$ of $\Lambda$.
\biglf
Note that both Lemma \RSref{Lemsigma} and \eref{eqrhoEB} furnish
similar error bounds with associated convergence rates,
and they might have the same behavior if the well-posedness norm
$\|.\|_{WP}$ coincides with  $\|.\|_{L_\infty(\overline\Omega)}$.
We shall have a closer look at this now, delaying experiments to
Section \RSref{SecNumRes}.
\biglf
To find a connection between $\rho_{\Lambda_N}$ and $\sigma_{\Lambda_N}$
in well-posed situations, we assume that the
well-posedness norm $\|.\|_{WP}$ of \eref{equWP} can be  written
as
\bql{eqWPgen}
\|u\|_{WP}=\displaystyle{\sup_{\mu\in M}|\mu(u)|,\;u\in \calh   } 
\eq
with a compact set $M\subset\calh^*$. This yields \eref{eqWP}
when taking $M$ as the set of delta functionals on $\overline\Omega$.
The generalization of \eref{eqrhodef} then is
\bql{eqrhoMdef}
\rho_{M,\Lambda_N}
:=\displaystyle{\max_{\mu \in M}P_{\Lambda_N}(\mu)   }. 
\eq
\begin{Lemma}\RSlabel{Lemrhosig}
  Assuming \eref{equWP} and \eref{eqWPgen}, we have
  $$
\rho_{M,\Lambda_N}\leq C_{WP}\sigma_{\Lambda_N}.
  $$
\end{Lemma}
\begin{proof}
The maximum in \eref{eqrhoMdef} is attained at some $\tilde \mu_N\in M$
and then
$$
 \begin{array}{rcl}
 \rho_{M,\Lambda_N}
  &=&
  P_{\Lambda_N}(\tilde \mu_N)\\
  &=&\dist(\tilde \mu_N,\span \Lambda_N) \\
  &=&
  \displaystyle{
    \left\|\tilde \mu_N-\sum_{j=1}^N(\tilde \mu_N,\mu_j)\mu_j\right\|_{\calh^*}}.  
\end{array}
$$
 Applied to an arbitrary $u\in\calh$ this yields
 $$
 \begin{array}{rcl}
   \left|\tilde \mu_N(u)-\sum_{j=1}^N(\tilde \mu_N,\mu_j)\mu_j(u)\right|
   &=&
   \left|\tilde \mu_N(u)-\tilde \mu_N(\tilde u_N)\right|\\
   &\leq &
   \sup_{\mu\in M}\left|\mu(u)-\mu(\tilde u_N)\right|\\
   &=& \|u-\tilde u_N\|_{WP}\\
   &\leq& C_{WP} \sup_{\lambda\in \Lambda} |\lambda(u)-\lambda(\tilde u_N)|\\
   &\leq & C_{WP} \sigma_{\Lambda_N}\|u\|_\calh
 \end{array}
 $$
 where the last line follows like \eref{eqpwmu}
 using a generalization of the argument within \eref{eqpwerrbnd}.
\end{proof}

\section{Greedy Method}\RSlabel{SecGM}
Given a set $\Lambda_N:=\{\lambda_1,\ldots,\lambda_N\}$,
the {\em $P$-greedy algorithm}
\RScite{DeMarchi-et-al:2005-1,santin-haasdonk:2017-1}
in its abstract form defines $\lambda_{N+1}$ recursively via
\bql{eqGdef} 
P_{\Lambda_N}(\lambda_{N+1})=\sup_{\lambda\in \Lambda}P_{\Lambda_N}(\lambda).
\eq
In the context of Reduced Basis Methods in Hilbert spaces,
this algorithm coincides
with the one in \cite{maday-et-al:2002-1,binev-et-al:2011-1,%
buffa-et-al:2012-1,devore-et-al:2013-1}. The main difference is that
we work in the dual here, focusing on a single class of PDEs
in applications instead of a parametrized family.
A first application of Reduced Basis Methods
within a Reproducing Kernel Hilbert Space 
setting is in \RScite{chen-et-al:2016-1}, implementing a
greedy method based on discrete
least-squares, not on the dual of the basic Hilbert space of functions. 
\biglf
The cited papers provide a useful error analysis of the method that we
state here for completeness. The main ingredient is
the {\em Kolmogorov $N$-width}
$$
d_N(\Lambda):=\inf_{\hbox{all } H_N}\sup_{\mu\in
  \Lambda}\inf_{\lambda\in H_N}
\|\lambda-\mu\|_{\calh^*}=\inf_{\hbox{all } H_N}\sup_{\mu\in
  \Lambda}\dist(\mu,H_N)_{\calh^*}
$$
where the first infimum is taken over all $N$-dimensional subspaces $H_N\subseteq
\calh^*$, not only those spanned by $N$ functionals from $\Lambda$.
Then \RScite{devore-et-al:2013-1} proves
$$
\sigma^2_{2N}(\Lambda)\leq 2 d_N(\Lambda),
$$
while \RScite{binev-et-al:2011-1} has
$$
\sigma_{\Lambda_N}\leq C(\alpha)N^{-\alpha} \hbox{ for }  n\in \N,
\hbox{ if }  d_{N}(\Lambda)\leq C'(\alpha)N^{-\alpha} \hbox{ for }  n\in \N
$$
with suitable constants. This links the behaviour of the
$P$-greedy algorithm to Kolmogoroff $N$-widths. 
The paper \RScite{santin-haasdonk:2017-1} exploits this connection
for function recovery
by interpolation, while this paper extends \RScite{santin-haasdonk:2017-1}
to classes of operator equations, including PDE solving.
\biglf
Consequently, the greedy method converges
roughly like the Kolmogoroff $N$-widths.
Asymptotically, there are no better choices for selecting $N$ functionals
out of $\Lambda$.
Note that this works for all well-posed problems
stated in abstract form in Hilbert space
via infinitely many constraints. The differential and boundary operators
can be easily generalized.
\biglf
On the downside, the literature does not provide much information
about the Kolmogoroff $N$-widths in such situations. This is why
on has to look at special cases. We postpone this to Section
\RSref{SecSobCas}.
\biglf
If we run the greedy algorithm numerically, we
should get very good candidates for reduced bases.
Choosing them with additional orthogonality properties will
then bring us close to Proper Orthogonal Decomposition methods.
The next section will explain how to do that. 
\subsection{Implementation}\RSlabel{SecImpl}
Throughout, we assume that the Hilbert space $\calh$ has a kernel $K$
such that inner products in $\calh^*$ can be numerically
calculated via
$$
(\lambda,\mu)_{\calh^*}=\lambda^x\mu^yK(x,y) \fa \lambda,\,\mu\in \calh^*
  $$
where the upper index stands for the variable the functional acts on.
In case of Sobolev spaces $W_2^m(\R^d)$ with $m>d/2$, we use the standard
Whittle-Mat\'ern kernel
$$
K_{m,d}(x,y)=\|x-y\|_2^{m-d/2}K_{m-d/2}(\|x-y\|_2),\;x,y\in\R^d
$$
with the modified Bessel function $K_\nu$ of second kind.
In what follows, we treat $\Lambda$ as being very large and finite,
but extensions to infinite compact $\Lambda$ will be possible
if functions on $\Lambda$ are discretized somehow. 
\biglf
A direct way to assess the generalized Power Function at $\lambda$ for
given $\Lambda_N:=\{\lambda_1,\ldots,\lambda_N\}$ is to use its definition
via the approximation problem
$$
\displaystyle{  \inf_{\alpha\in\R^N}\left\|\lambda-
  \sum_{j=1}^N\alpha_j\lambda_j\right\|^2_{\calh^*}   }=P_{\Lambda_N}^2(\lambda).
$$
We assume that we have turned $\lambda_1,\ldots,\lambda_N$
into an $\calh^*$-orthonormal
basis $\mu_1,\ldots,\mu_N$ already, and then the solution is
\bql{eqGPFll}
P_{\Lambda_N}^2(\lambda)=\displaystyle{ \left\|\lambda-
\sum_{j=1}^N(\lambda,\mu_j)_{\calh^*} \mu_j\right\|^2_{\calh^*}
=(\lambda,\lambda)_{\calh^*} -\sum_{j=1}^N(\lambda,\mu_j)^2_{\calh^*}}.
\eq
Furthermore, we store the orthonormalization
system for getting the $\mu_j$ from the $\lambda_j$ as
\bql{eqmufromlambda}
\mu_k=\displaystyle{\sum_{j=1}^{k-1}c_{kj}\mu_j+c_{k,k}\lambda_k,\;1\leq k\leq
  N.   }
\eq
Besides the $N\times N$ triangular matrix $C$, we store the values
$\{(\lambda,\lambda)_{\calh^*}\}_{\lambda\in\Lambda}$ and $\{(\lambda,\mu_k)_{\calh^*}\}_{\lambda\in\Lambda}$ 
for $k=1,\ldots,N$. These make up the method's bulk storage of order
$(N+1)|\Lambda|$. Note that $\{(\lambda,\lambda)_{\calh^*}\}_{\lambda\in\Lambda}$
simplifies considerably for translation-invariant kernels.
\biglf
After maximizing \eref{eqGPFll} over all $\lambda\in\Lambda$,
we assume to have some $\lambda_{N+1}$ at which a nonzero maximum is attained.
This can then not be one of the old $\lambda_j$, and we retrieve the
values $(\lambda_{N+1},\mu_k)_{\calh^*}$ from what we have. We now orthonormalize 
$$
\mu_{N+1}=\displaystyle{\sum_{j=1}^{N}c_{N+1,j}\mu_j+c_{N+1,N+1}\lambda_{N+1}}
$$
via $c_{N+1,j}=-(\lambda_{N+1},\mu_k)_{\calh^*}c_{N+1,N+1},\;1\leq k\leq N$ and
$$
\begin{array}{rcl}
1=(\mu_{N+1},\mu_{N+1})_{\calh^*}&=&\displaystyle{  c^2_{N+1,N+1}
\left((\lambda_{N+1},\lambda_{N+1})_{\calh^*}- \sum_{k=1}^N(\lambda_{N+1},\mu_k)_{\calh^*}^2 \right)}\\
&=&c^2_{N+1,N+1}P^2_{\Lambda_N}(\lambda_{N+1})
\end{array} 
$$
to update the $C$ matrix. Finally
$$
\begin{array}{rcl}
  (\lambda,\mu_{N+1})_{\calh^*}
  &=&\displaystyle{\sum_{j=1}^{N}c_{N+1,j}(\lambda,\mu_j)_{\calh^*}+c_{N+1,N+1}(\lambda,\lambda_{N+1})_{\calh^*}}
\end{array}
$$
can be calculated from what we have, if we
first calculate all
$\{(\lambda,\lambda_{N+1})_{\calh^*}\}_{\lambda\in\Lambda}$ and overwrite them
with $\{(\lambda,\mu_{N+1})_{\calh^*}\}_{\lambda\in\Lambda}$ after use.
This extends the {\em Newton basis}
technique in \RScite{pazouki-schaback:2011-1}
to general functionals.
\subsection{Bases and Postprocessing}\RSlabel{SecGrePost}
Useful reduced bases for PDE solving are the
Riesz representers $v_{\mu_k}$ of the $\mu_k$,
being orthonormal in Hilbert space.
From Riesz representers
$
v_{\lambda_j}(\cdot)=\lambda_j^xK(x,\cdot)
$
of the $\lambda_j$, we can calculate them
recursively via
\bql{eqvcalc}
v_{\mu_k}(x)=\displaystyle{\sum_{j=1}^{k-1}c_{kj}v_{\mu_j}(x)+c_{k,k}\,
  v_{\lambda_k}(x),\;1\leq k\leq
  n  }
\eq
on whatever point sets we like, using our triangular matrix $C$. Via
\eref{eqGPFll} in the form
\bql{eqPKv}
P^2_{\Lambda_N}(\delta_x)=K(x,x)-\sum_{j=1}^Nv^2_{\mu_j}(x),
\eq
this allows to calculate the pointwise error bounds
\eref{eqpwerrbnd} described in Section
\RSref{SecEA} explicitly, up to the term $\|u\|_\calh$.
\biglf
In view of Proper Orthogonal Decomposition, one can also
apply a Singular Value Decomposition to the partial
Gramian matrix with entries $(\lambda_j,\lambda_k)_{\calh^*},\;1\leq j,k\leq N$
and construct a different $\calh$-orthonormal basis.
This possibility is not pursued here, because we want to keep
the recursive structure of the algorithm. In view
of \RScite{santin-schaback:2016-1}, this basis may be closer to
what happens for Kolmogoroff $N$-widths, because the spaces for the latter
will not necessarily have a recursive structure.
\biglf
Like in reduced basis techniques, the orthonormal
bases are a simple tool to
solve a variety of similar problems, namely all
problems of the form \eref{eqLufg}. If data $\lambda_j(u)$
are known for an unknown function $u$, we go over to $\mu_j(u)$
via the triangular system
\eref{eqmufromlambda} and then form the optimal projection
\bql{equNtilde}
\tilde u_N:=\displaystyle{\sum_{j=1}^N\mu_j(u)v_{\mu_j}   } 
\eq
using the $\calh$-orthonormal basis we have constructed. 
\biglf
Note that this algorithm
neither stores nor solves a large $|\Lambda|\times |\Lambda|$
system. It works ``on--the--fly''. For $N$ steps, \
storage is of order ${\cal O}(N\cdot|\Lambda|+N^2)$ and calculations
are of order ${\cal O}(N^2\cdot|\Lambda|)$. 
For a given accuracy requirement, the number $N$ of steps
will depend on the Kolmogoroff $N$-width for the set $\Lambda$
and the Hilbert space $\calh^*$.
\biglf
There are no square systems to be solved.
Instead, there are $N$ orthonormalization steps on vectors of length
$|\Lambda|$ that may require standard stabilization
precautions for the basic Gram-Schmidt technique. The resulting triangular
$N\times N$ matrix $C$ is not explicitly
inverted, but used via \eref{eqmufromlambda}
to transform input data
in terms of the $\lambda_j$ functionals into data in terms of the
orthonormalized functionals $\mu_j$. The
error behavior of this is comparable to backsubstitution
after an $LR$ or $QR$ factorization, but it will pay the price
when some of the $\lambda_j$ are strongly correlated.
This will be unavoidable for large $N$, but
\eref{eqPFdef} and \eref{eqGdef} show that $\lambda_{N+1}$ will be
kept away from the zeros $\lambda_1,\ldots,\lambda_N$ of $P_{\Lambda_N}$
and the space they span, 
by construction.

\subsection{Extended Greedy Method}\RSlabel{SecVGM}
One can get somewhat closer to the error analysis in Section
\RSref{SecEA} and in particular to \eref{eqpwmu} by a modification of the
selection strategy of functionals.
Given a set $\Lambda_N:=\{\lambda_1,\ldots,\lambda_N\}$,
and a set $M=\{\mu_1,\ldots,\mu_M\}$ with nonempty intersection,
we calculate two maxima
$$
\begin{array}{rclcl}
\tilde \lambda&:=&\arg \sup _{\lambda\in \Lambda}P_{\Lambda_N}(\lambda)&=&\sigma_{\Lambda_N}\\
\tilde \mu&:=&\arg \sup _{\mu \in M }P_{\Lambda_N}(\mu)&=:&\rho_{\Lambda_N}
\end{array}
$$
and set
$$
\lambda_{N+1}:=\left\{\begin{array}{rcl}
\tilde \mu & & \hbox{ if } \mu \in \Lambda\cap M\\ 
\tilde \lambda & & \hbox{ else. } 
\end{array}     \right.
$$
This can be called an {\em extended $P$-greedy algorithm}. It tries to
keep some additional control of $\rho_{M,\Lambda_N}$ from \eref{eqrhoMdef}
by selecting functionals from $\Lambda\cap M$ whenever $\rho_{M, \Lambda_N}$ is
attained on them.
\biglf
For solving Dirichlet problems, the set $M$ will consist
of delta functionals in $\overline \Omega$, the intersection
of $M$ and $\Lambda$ being $\Lambda_2$, the set of functionals
for Dirichlet boundary values. The new technique  will make sure that if $P_{\Lambda_N}(M)$
attains its maximum on the boundary, the corresponding functional
is preferred over the functional where $P_{\lambda_N}(\Lambda)$
attains its maximum. 
\biglf
Now for some inplementation details for the Dirichlet case.
Define the set $Z$ of boundary points via
$$
\{\delta_z\;:\;z\in Z\}=\Lambda_2
$$
and add some other point set $Y\subset\overline\Omega$ to get
$$
M:=\{\delta_y\;:\;y\in Y\}\cup \Lambda_2=\{\delta_x\;:\;x\in Y\cup Z\}.
$$
We need the additional value $\|P_{\Lambda_N}\|_{\infty,Y}$, while the
usual Greedy method provides $\|P_{\Lambda_N}\|_{\infty,Z}$ as part of the
calculation of $\rho_{\Lambda_N}$.
If we use \eref{eqPKv} on $Y$, we have the necessary data, but for that we
have to evaluate \eref{eqvcalc} on $Y$ as well, and have to store
the values of the orthonormal basis on $Y$. This requires an additional
storage of size $|Y|\cdot N$ for $N$ steps, and additional calculations
of order $|Y|\cdot N^2$. On can choose  $|Y|$ smaller than $|\Lambda_1|$
to keep the complexity at bay.

\section{Sobolev case}\RSlabel{SecSobCas}
We now go back to the example at the beginning.
We need information
on the Kolmogoroff $N$-width
$$
  d_N(\Lambda)
=\displaystyle{  \inf_{\hbox{all } H_N}\sup_{\mu\in
  \Lambda}\dist(\mu,H_N)_{{W_2^m(\Omega)}^*}}
$$
  where the infimum is taken over all $N$-dimensional subspaces
  $H_N$ of $\calh^*={W_2^m(\Omega)}^*$,
  and the set $\Lambda=\Lambda_1\cup\Lambda_2\subset \calh^*$
  is formed by \eref{eqLamdef}.
  There should be an $m$- and $d$-dependent
  decay rate $\kappa(m,d)$ in the sense
  $$
d_N(\Lambda)\leq CN^{-\kappa(m,d)} \hbox{ for } N\to\infty 
  $$
but no explicit results on this were found yet.
\biglf
If we restrict attention to spaces $H_N$ being generated by $N$
functionals from $\Lambda=\Lambda_1\cup\Lambda_2\subset \calh^*$,
we only get an upper bound for $d_N(\Lambda)$, and we
do not know the splitting $N=N_1+N_2$ if an optimal choice of
$N$ functionals from $\Lambda$  
takes $N_1$ functionals out of $\Lambda_1$ and $N_2$
functionals out of $\Lambda_2$. 
The rest of the chapter will give some arguments
supporting the hypothesis
\bql{eqkappahypo}
\kappa(m,d)\geq \dfrac{m-2-d/2}{d}
\eq
that will be observed in the numerical behavior of the $P$-greedy method
in Section \RSref{SecNumRes}.
\biglf
If $\Lambda$ is the union of two disjoint compact sets
$\Lambda_1$ and $\Lambda_2$, then
$$
\begin{array}{rcl}
  d_N(\Lambda)
&=&\displaystyle{  \inf_{\hbox{all } H_N}\sup_{\mu\in
  \Lambda}\dist(\mu,H_N)_{\calh^*}}\\
  &=&\displaystyle{  \inf_{\hbox{all } H_N}\max\left(
    \sup_{\mu\in
      \Lambda_1}\dist(\mu,H_N)_{\calh^*},\sup_{\mu\in
      \Lambda_2}\dist(\mu,H_N)_{\calh^*}\right)
  }\\
 &\leq& \displaystyle{\inf_{H_{N_1},H_{N_2},N_1+N_2\leq N,\,H_N=H_{N_1}+H_{N_2}}
\max(d_{N_1}(\Lambda_1),d_{N_2}(\Lambda_2))}
\end{array} 
$$
where $H_{N_i}$ is used for approximation of $\Lambda_i$,
and the sum of spaces is direct.
This {\em splitting argument} is similar to the technique
in \RScite{franke-schaback:1998-2a}. However, we do not know a priori
how the greedy algorithm selects functionals from either set,
and how the dimension splits into $N_1+N_2\leq N$.
\biglf
We first aim at $d_N(\Lambda_1)$ in the space ${W_2^m(\Omega)}^*$.
We have
$$
\begin{array}{rcl}
  d_N(\Lambda_1)
  &=& \displaystyle{ \inf_{\hbox{ all } H_N}\sup_{x\in\Omega}\inf_{\lambda\in
    H_N}\|\lambda-\delta_x\circ L\|_{{W_2^m(\Omega)}^*}}\\ 
\end{array}
$$
and can majorise it by choosing $N$ asymptotically uniformly placed
points $x_1,\ldots x_N$
in $\overline\Omega$ at fill distance $h_\Omega$ and taking
$H_N$ to be the span of
the corresponding functionals $\delta_{x_j}\circ L$. Then
$$
\begin{array}{rcl}
  d_N(\Lambda_1)
  &\leq& \displaystyle{\sup_{x\in\Omega}\inf_{\alpha\in\R^N}
   \|\delta_x\circ L -\sum_{j=1}^N\alpha_j\delta_{x_j}\circ L\|_{{W_2^m(\Omega)}^*}}\\ 
\end{array}
$$
holds and we can expect
$$
\begin{array}{rcl}
  d_N(\Lambda_1)
  &\leq& C\,\displaystyle{\sup_{x\in\Omega}\inf_{\alpha\in\R^N}
   \|\delta_x-\sum_{j=1}^N\alpha_j\delta_{x_j}\|_{{W_2^{m-2}(\Omega)}^*}}\\
  &\leq& C\,h_\Omega^{m-2-d/2}\leq C\,N^{-(m-2-d/2)/d}
\end{array}
$$
due to standard results on error bounds for interpolation
\RScite{wendland:2005-1}, and with generic constants depending on $m$, $d$
and the domain. 
\biglf
On the boundary $\Gamma$, we can argue similarly to expect
$$
d_N(\Lambda_2)\leq C\,h_\Gamma^{m-d/2}\leq C\,N^{-\frac{m-d/2}{d-1}}
$$
for a fill distance $h_\Gamma$ on the boundary,
either by working in ${W_2^m(\Omega)}^*$ directly,
or via trace theorems, which would
give the same rate due to $m-1/2-(d-1)/2=m-d/2$.
\biglf
If we now consider splittings $N=N_1+N_2$, we roughly have
the upper bound
$$
d_N(\Lambda)\leq
C\max\left(N_1^{-(m-2-d/2)/d},N_2^{-\frac{m-d/2}{d-1}} \right)
$$
and the crude split $N_1\approx N/2\approx N_2$
will already support \eref{eqkappahypo}.
\biglf
For purposes of asymptotics, we can minimize
the sum of the above quantities instead of the maximum.
Using standard optimization arguments under the constraint
$N_1+N_2\leq N$, the result after some calculations is that one should expect
$$
c_1\frac{m-2-d/2}{d}N_1^{-\frac{m-2+d/2}{d}}
=c_2\frac{m-d/2}{d-1}N_2^{-\frac{m-1+d/2}{d-1}}
$$
with constants depending on the domain and the space dimension,
but not on $N$ and $m$. 
For the typical case $m=4,\;d=2$ this implies
$N_2\approx N_1^{3/8}$. For large $m$ and
$d$ we get $N_1\approx N_2$, i.e. the domain and the boundary require roughly
the same degrees of freedom. This is no miracle, because the volumes
of balls
of high dimension are increasingly accounted for by the boundary layer.
In general, the above argument implies $N_2\leq c\,N_1$ for
$N$ large enough, and then \eref{eqkappahypo} will hold.
Section \RSref{SecNumRes} will reveal that the greedy
algorithm selects unexpectedly small values of $N_2$, and that the
hypothesis \eref{eqkappahypo} is supported.
\section{Numerical Results}\RSlabel{SecNumRes}
No matter which PDE examples are selected, the most interesting
question is the behavior of the greedy method as a function
of the $N$ steps it takes. If $\Lambda$ is chosen large enough
but still finite, the crucial quantities are
$\sigma_{\Lambda_N}$ and $\rho_{\Lambda_N}$
as defined in \eref{eqsigmadef} and  \eref{eqrhodef}.
By Section \RSref{SecGM} we can expect that $\sigma_{\Lambda_N}$
decays at an optimal rate comparable to the Kolmogoroff $N$-width
with respect to $\Lambda$, but since the latter is still unknown, we
can not yet assess how close we come to it. Since $\rho_{\Lambda_N}$
is controlling the error in the sup norm, without being usable
for greedy refinement, and in view of Lemma
\RSref{Lemrhosig}, we would like to confirm that
$\rho_{\Lambda_N}$ decays as fast  as $\sigma_{\Lambda_N}$.
And, the decay rates should improve with smoothness of the functions
in the basic Hilbert space, i.e. with $m$ if we work
in $W_2^m(\R^d)$, our hypothesis being
\eref{eqkappahypo}. Another interesting question is how the greedy
method chooses between boundary functionals
from $\Lambda_2$ and domain functionals from $\Lambda_1$,
and whether the corresponding points are roughly uniformly
distributed in both cases. Finally, the shape and the behaviour
of the basis functions $v_{\mu_j}$ should be demonstrated.
\subsection{Observations for varying $N$}\RSlabel{SecNumResObsCVN}
We start on the 2D unit disk, with $L$ being the Laplace operator,
carrying the greedy method out for up to 500 steps, offering 17570
functionals for $\Lambda_1$ and 150 functionals for $\Lambda_2$.
The Hilbert space will be $W_2^m(\R^2)$, but we fix $m>2+d/2$ first,
to study the behaviour of the greedy method for varying the number $N$
of steps.
Classical results on kernel-based interpolation lets us
expect rates
for $\sigma_{\Lambda_N}$ and $\rho_{\Lambda_N}$ that are
determined by fill distances. If $h_\Gamma$ and $h_\Omega$ are fill
distances
for points on the boundary $\Gamma$ and in the domain $\Omega$,
we can compare $\sigma_{\Lambda_N}$
with plain interpolation of $\Delta u=f$
with smoothness $m-2$ and behavior
\bql{eqhDexp}
h_\Omega^{m-2-d/2}=h_\Omega^{m-3}\approx N_\Omega^{-\frac{m-2-d/2}{d}}=N_\Omega^{-\frac{m-3}{2}}
\eq
on the domain. 
The error on the boundary in $L_\infty$
should, if it were a
plain interpolation,
behave like  
\bql{eqhBexp}
h_\Gamma^{m-1/2-(d-1)/2}=h_\Gamma^{m-d/2}=h_\Gamma^{m-1}\approx N_\Gamma^{-\frac{m-d/2}{d-1}}=N_\Gamma^{-(m-1)},
\eq
if $N_\Gamma$ points are asymptotically equally spaced on the boundary.
But we do not know a priori how the greedy algorithm chooses
between boundary and domain functionals.
\biglf
The first experiments are for $m=4>2+d/2=3$, and we ignore the split
$N=N_\Omega+N_\Gamma$ in the beginning. The scale of the
Whittle-Mat\'ern kernel is chosen to be 1, and then both types of functionals
happen to have the same norm. Scaling changes the relation between
function value evaluation and Laplace operator evaluation and must be used
with care. We wanted to eliminate scaling effects for what follows,
in order to let the choice between boundary and domain functionals be
unbiased by possibly different norms.
\biglf
Figure \RSref{Fig020301500p3rates} shows 
$\sigma_{\Lambda_N}$ (left) and $\rho_{\Lambda_N}$ (right)
as functions of $N$, with the observed rates $-0.45$ and $-0.54$, respectively.
We see that
$\rho_{\Lambda_N}$ decays as fast  as $\sigma_{\Lambda_N}$ as
functions of $N$.
\begin{figure}[hbt]
\begin{center}
\includegraphics[height=4.0cm,width=6.0cm]{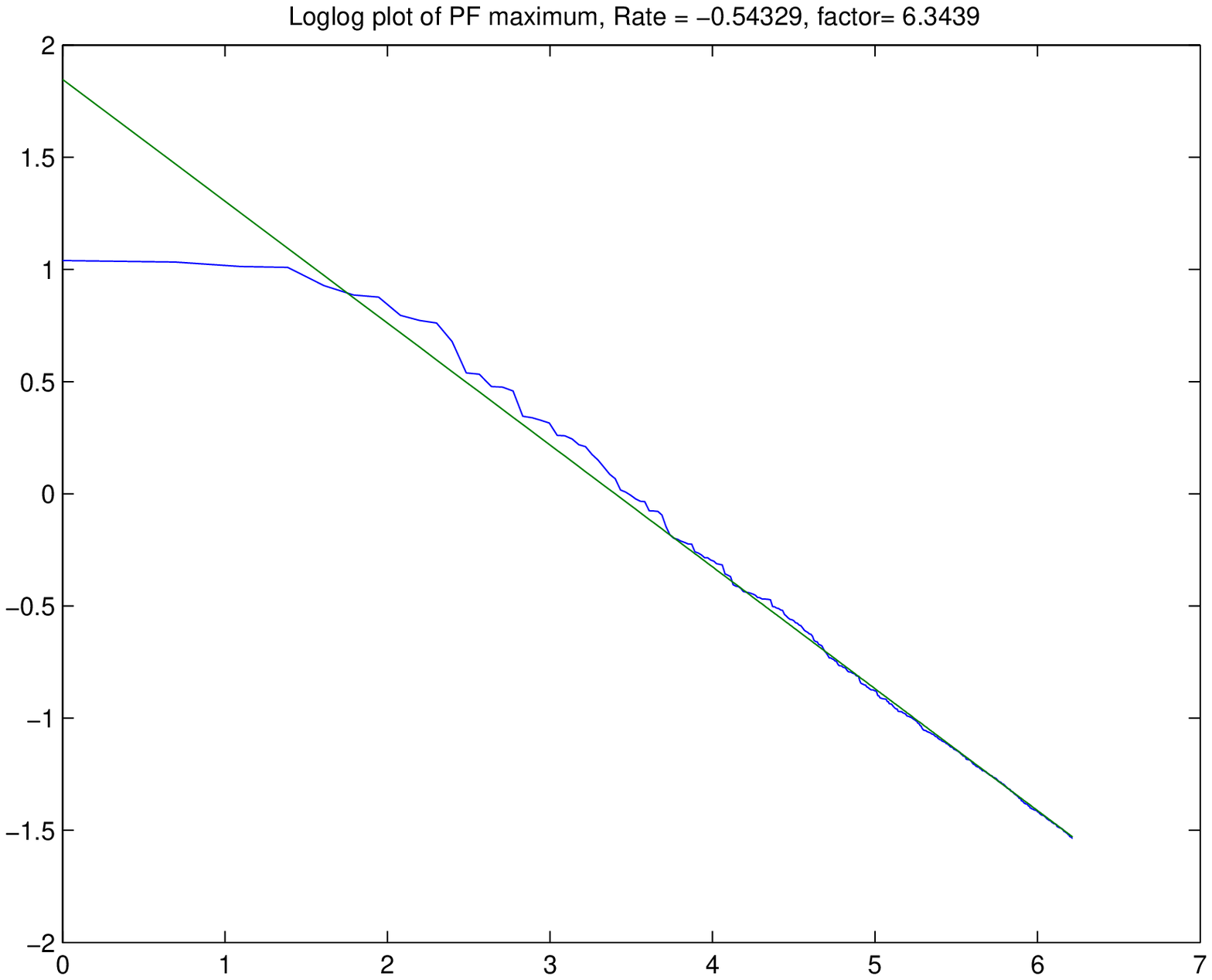} 
\includegraphics[height=4.0cm,width=6.0cm]{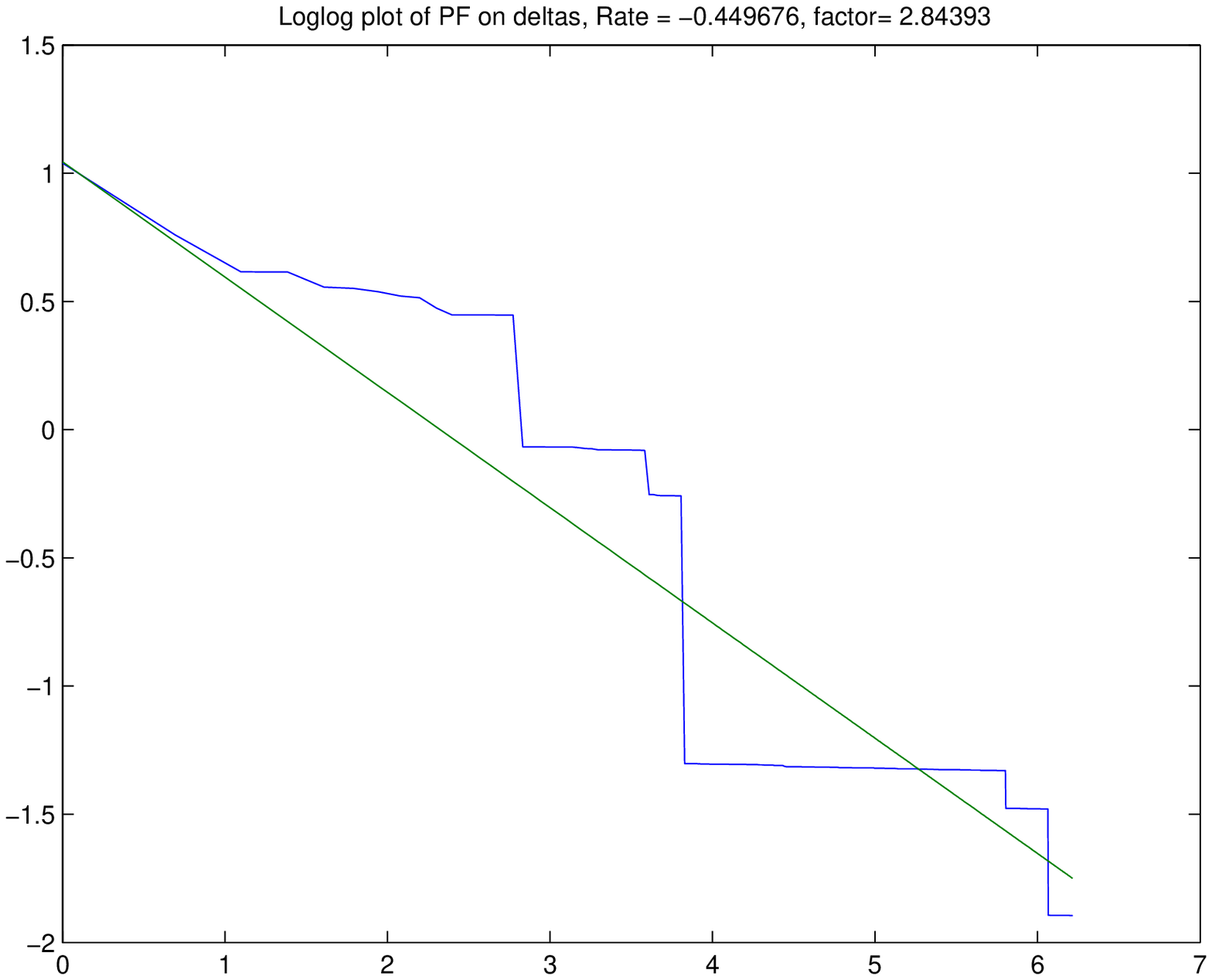} 
\end{center}
\caption{$\sigma_{\Lambda_N}$ and $\rho_{\Lambda_N}$
  and their rates as functions of $N$\RSlabel{Fig020301500p3rates}}
\end{figure} 
\biglf
The strange drops of $\rho_{\Lambda_N}$ at certain $N$ occur exactly when the fill
density $h_\Gamma$ on the boundary drops, namely
after the greedy method has chosen another boundary
functional.
This can be read off the two plots of Figure \RSref{Fig020301500p3both} that show both 
$\sigma_{\Lambda_N}$ and $\rho_{\Lambda_N}$ (left) and both
$h_\Gamma(N)$ and $h_\Omega(N)$ as functions of the $N$
functionals that were greedily selected.
Figure \RSref{Fig020301500p3PF} shows the strange fact that
the greedy method selects only rather few boundary functionals
compared to domain functionals (6 versus 494). It allows two large peaks
on the boundary, because it still fights for getting small on the domain
functionals.
These effects were observed in many other cases that we suppress here for
brevity.
Choosing different weights for domain and boundary functionals
makes some sense in view of well-posedness inequalities
like \eref{eqWP}, and then the effect will be less apparent.
\begin{figure}[hbt]
\begin{center}
\includegraphics[height=4.0cm,width=6.0cm]{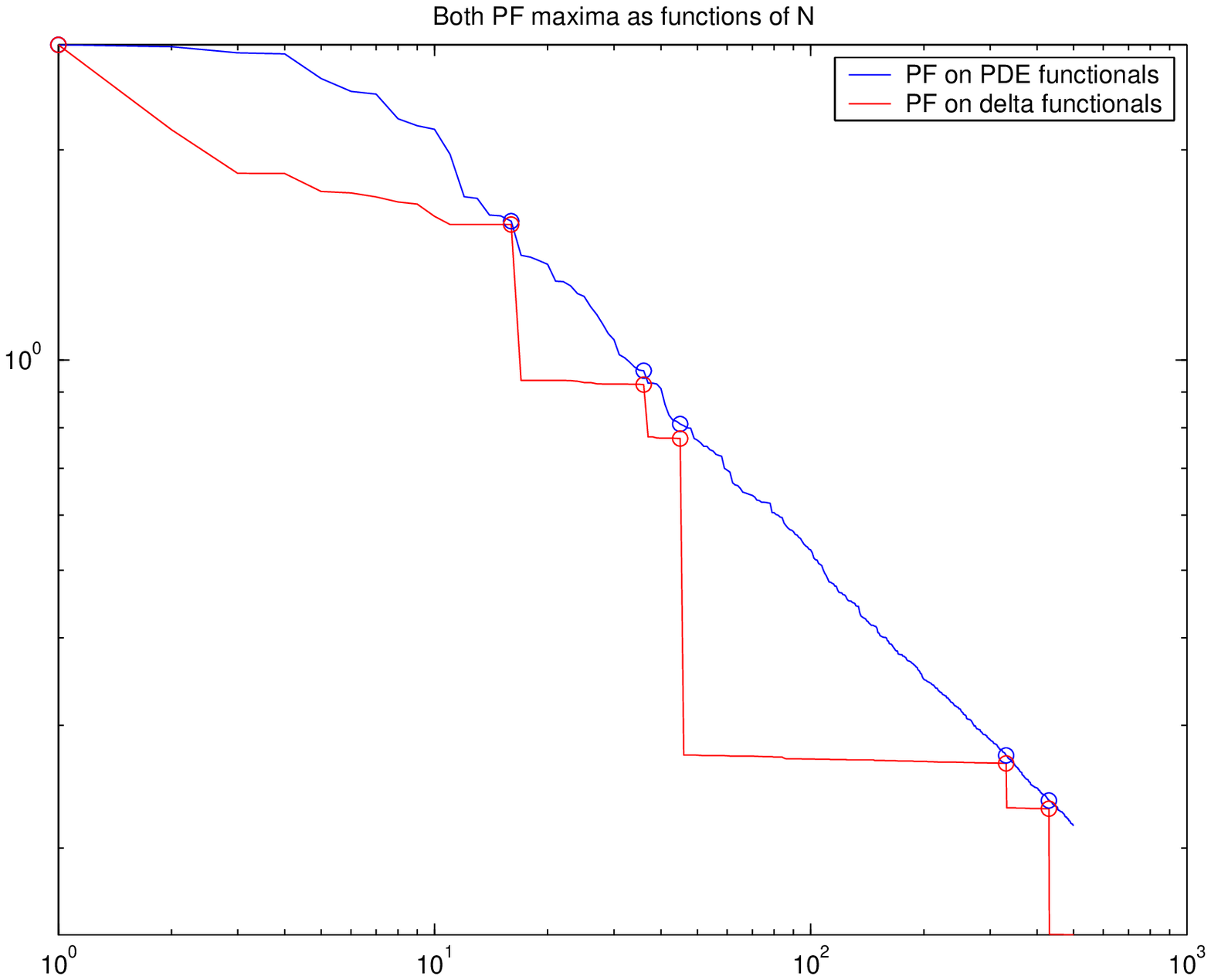}\\ 
\includegraphics[height=4.0cm,width=6.0cm]{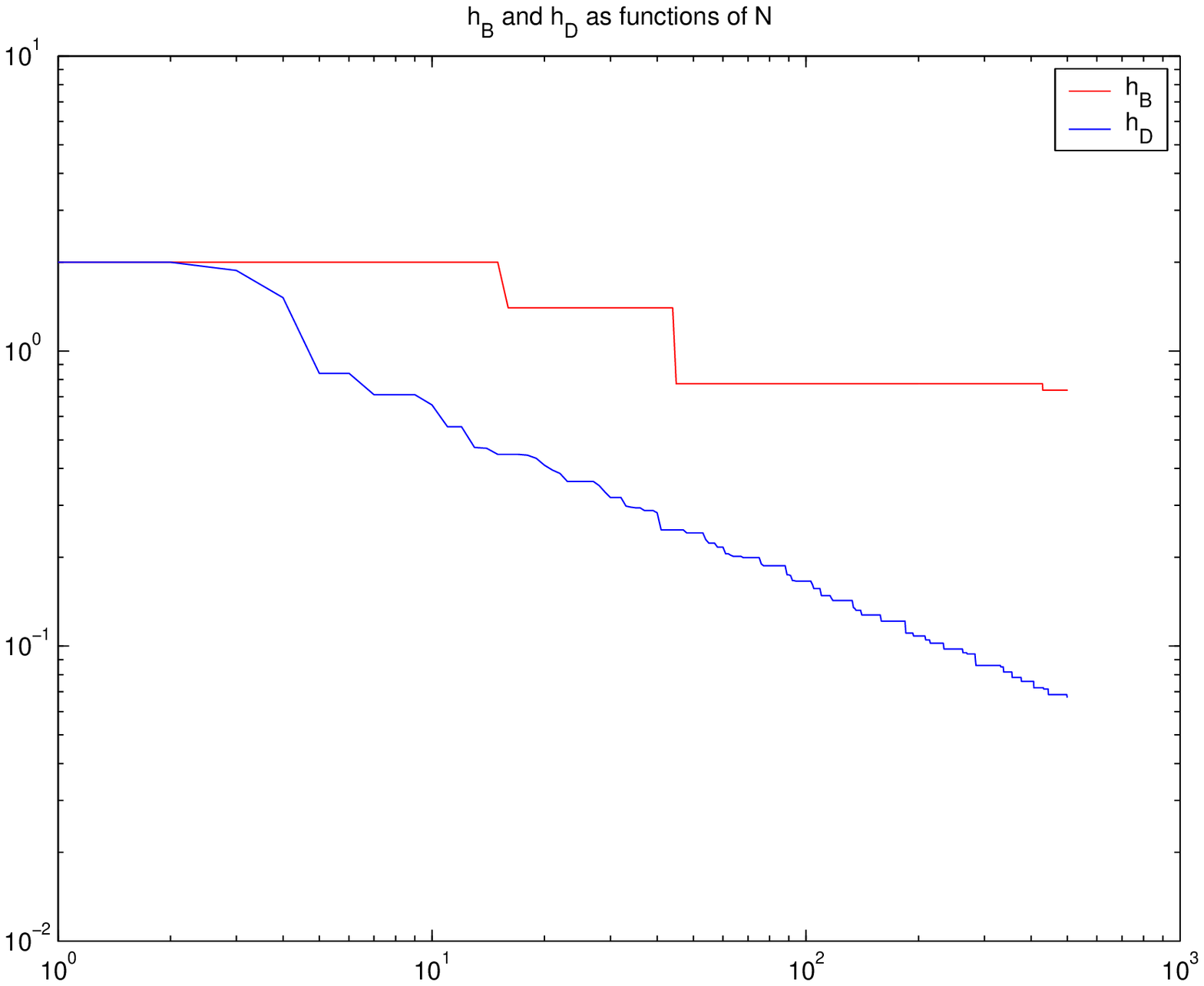} 
\end{center}
\caption{$\sigma_{\Lambda_N}$ and $\rho_{\Lambda_N}$ (top)
  compared with $h_\Gamma(N)$ and $h_\Omega(N)$ (bottom) 
  as functions of $N$. The drops of $\rho_{\Lambda_N}$
  and $h_\Gamma(N)$ occur at the same $N$.\RSlabel{Fig020301500p3both}}
\end{figure}
\biglf
Figure \RSref{Fig020301500p3PFN12} shows decays like 
$$
P_{\Lambda_N}\approx N_\Omega^{-0.53} \hbox{ and }  P_{\Lambda_N}\approx N_\Gamma^{-1.42}
$$
in comparison to the expectations in \eref{eqhBexp} and \eref{eqhDexp}
that suggest rates $-0.5$ and $-3$, respectively, for $m=4$.
This shows again that
the greedy method focuses on domain points, and is able to maintain a small
error on the boundary by adding ``domain'' functionals there and close to the
boundary.  
\begin{figure}[hbt]
\begin{center}
  \includegraphics[height=6.0cm,width=6.0cm]{%
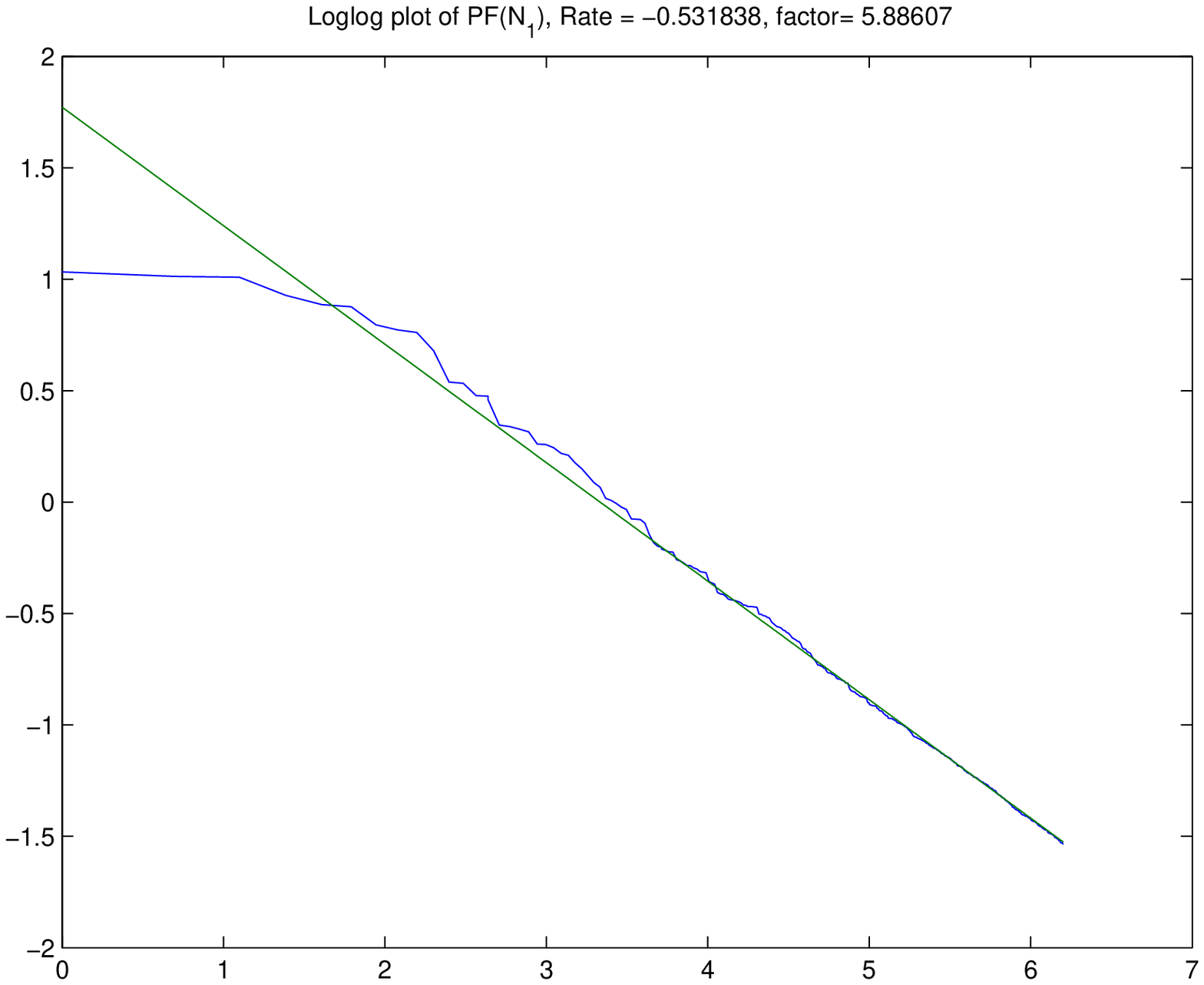} 
  \includegraphics[height=6.0cm,width=6.0cm]{%
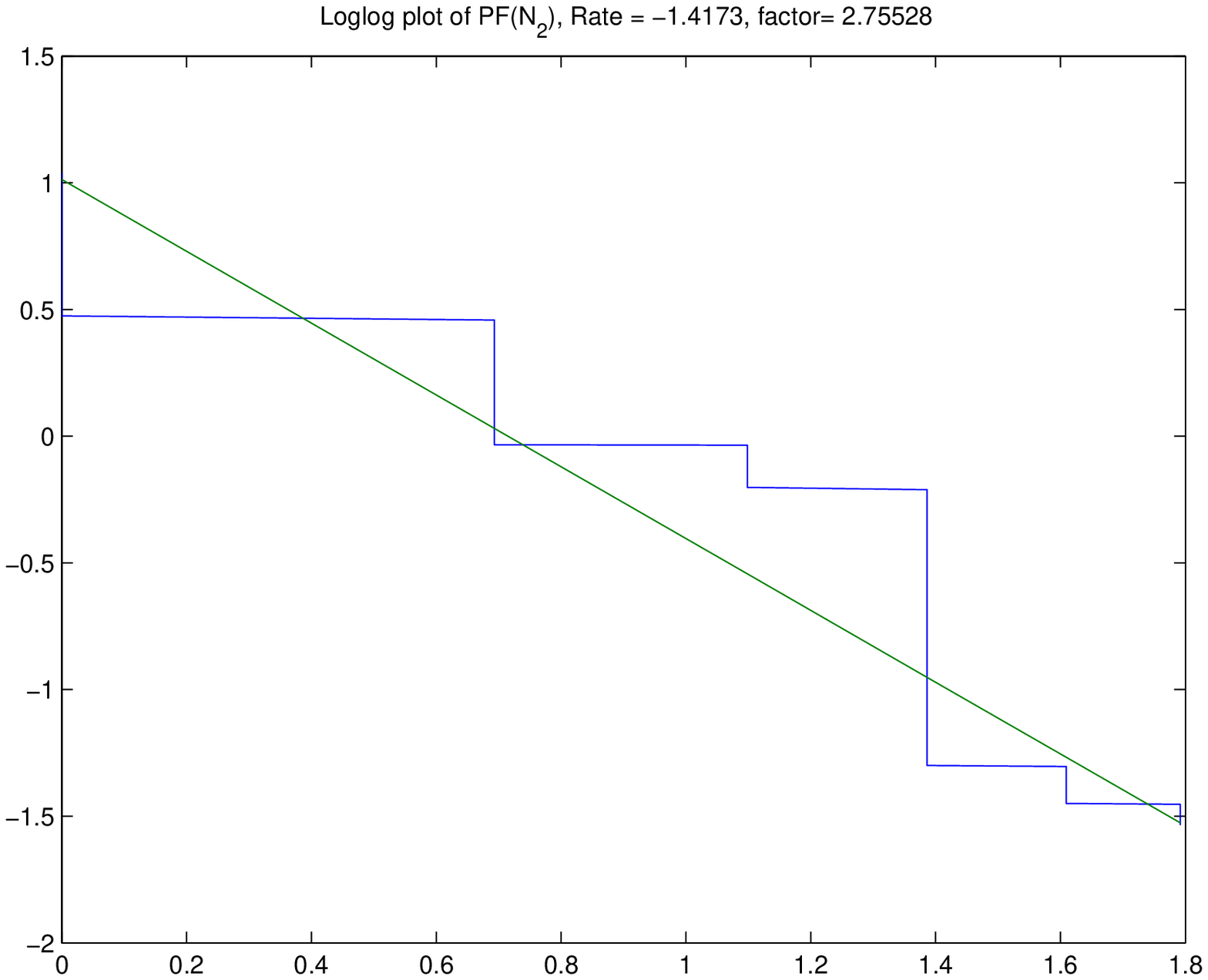} 
\end{center}
\caption{$P_{\Lambda_N}$ with $N=N_\Omega+N_\Gamma$
  as a function of $N_\Omega$ and $N_\Gamma$, respectively.
\RSlabel{Fig020301500p3PFN12}}
\end{figure} 

\begin{figure}[hbt]
\begin{center}
\includegraphics[height=4.0cm,width=12.0cm]{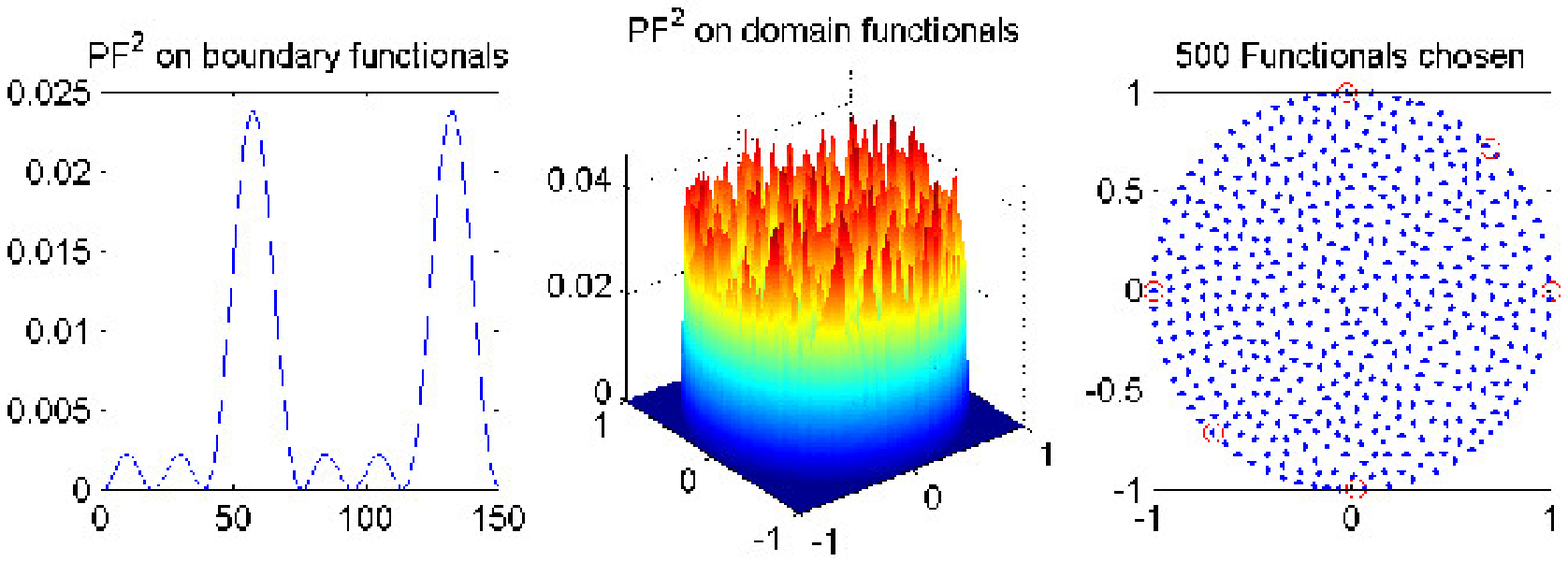}\\ 
\end{center}
\caption{$P^2_{\Lambda_{500}}$ on the boundary functionals, the domain
  functionals, and the selected 500 functionals
\RSlabel{Fig020301500p3PF}}
\end{figure} 
\biglf
If one looks at the Power Function on $\delta$ functionals all over the domain,
like for the $\rho_{\Lambda_N}$ calculation, the results are  in
Figure 
\RSref{Fig020301500p3NB}. It shows that for a better overall $L_\infty$ bound
it would be useful to pick boundary functionals at the boundary peaks of the
left plot. This calls for a variation of the greedy method that monitors 
$P_{\Lambda_N}$ on all $\delta$ functionals as well, and picks a boundary
functional as soon as the maximum is on the boundary. This will work for
Dirichlet problems, but not in general circumstances in Hilbert spaces.
See Sections \RSref{SecVGM}  for theory
and \RSref{SecNumResExtGM} for numerical results, respectively.
\begin{figure}[hbt]
\begin{center}
\includegraphics[height=4.0cm,width=12.0cm]{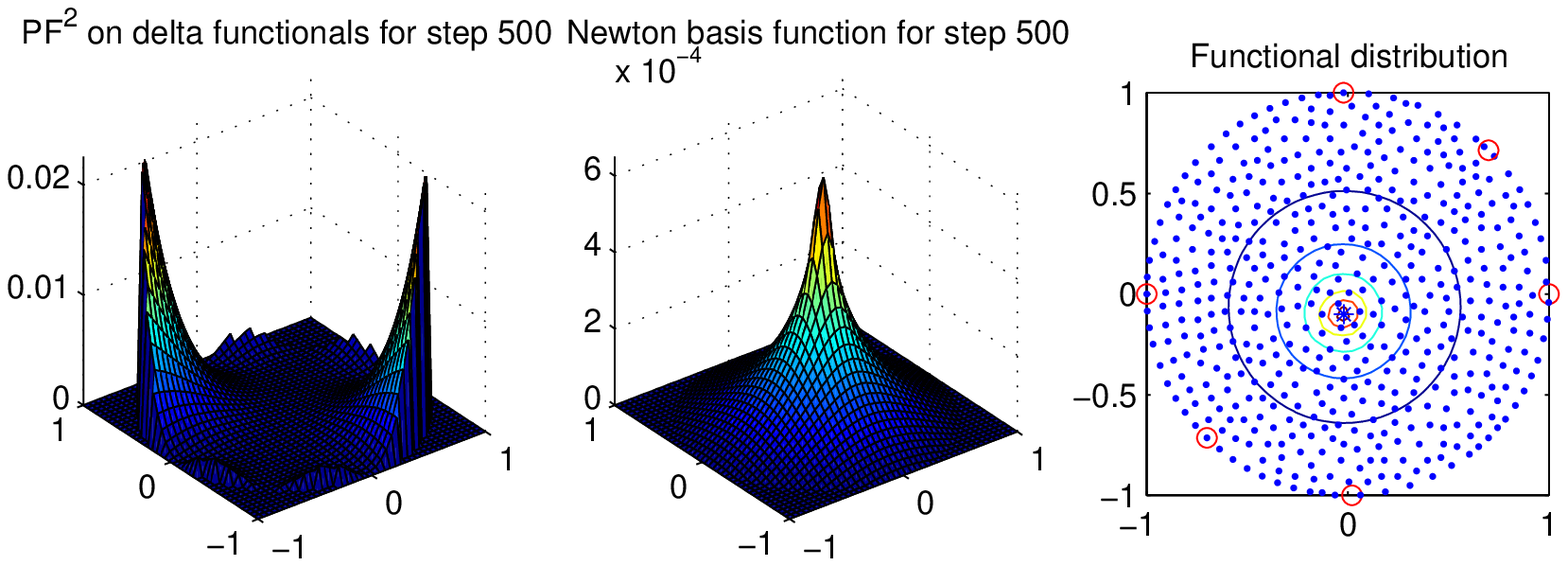}\\ 
\end{center}
\caption{$P^2_{\Lambda_{500}}(\delta_x)$ on the domain, the basis function
  $v_{\mu_{500}}$, and the selected 500 functionals, with the final selected
  domain functional marked with a blue asterisk in the center, the contours being those of
  $v_{\mu_{500}}$.
\RSlabel{Fig020301500p3NB}}
\end{figure} 
\biglf
The basis functions $v_{\mu_N}$ are orthonormal in Hilbert space but not in
$L_2$, decay with $N$, and show a sharp bell-shape for large $N$. Figure
\RSref{Fig020301500p3NB150} shows a case like Figure \RSref{Fig020301500p3NB},
but for $N=150$.
The new point, marked with an asterisk, is preferred over any boundary point by
the greedy method, though it is close to the boundary. The greedy method, as is,
does not need many $\delta$ functionals on the boundary, if it has
plenty of $\delta\circ L$ functionals on or near the boundary. 
\begin{figure}[hbt]
\begin{center}
\includegraphics[height=4.0cm,width=12.0cm]{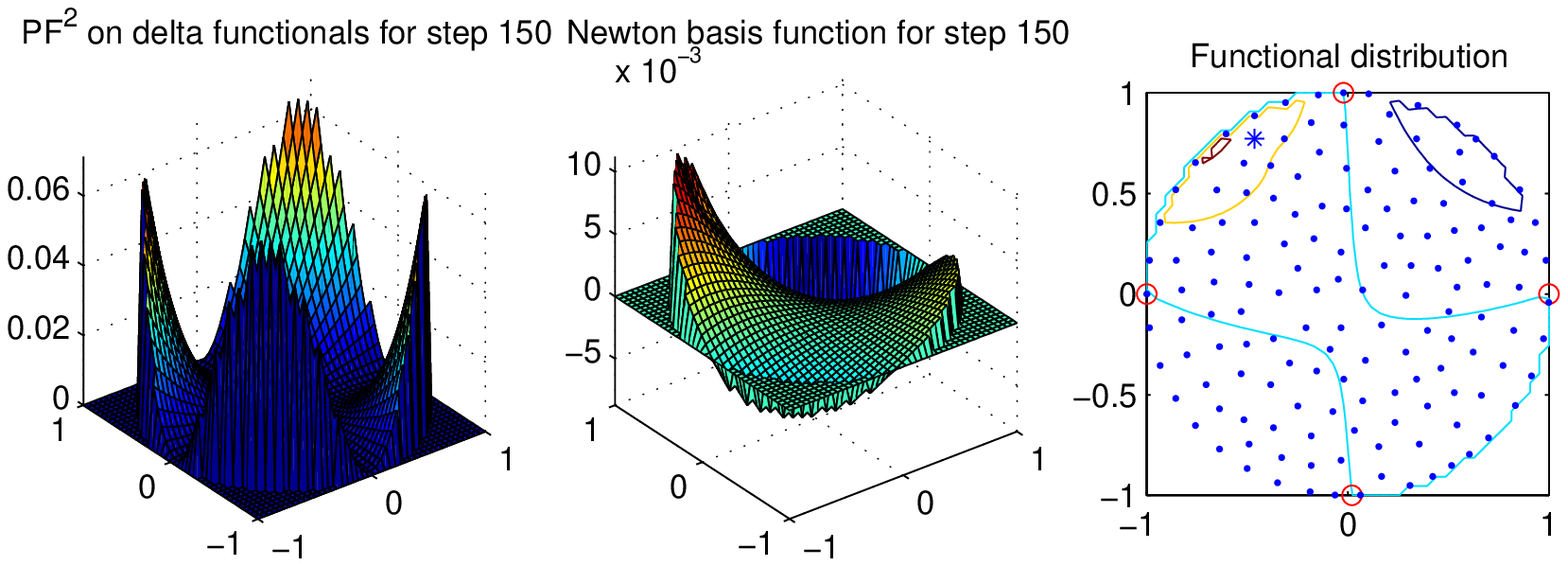}\\ 
\end{center}
\caption{$P^2_{\Lambda_{500}}(\delta_x)$ on the domain, the basis function
  $v_{\mu_{150}}$, and the selected 150 functionals, with the final selected
  domain functional marked with a blue asterisk in the northwest,
  the contours being those of
  $v_{\mu_{150}}$.
\RSlabel{Fig020301500p3NB150}}
\end{figure}

\biglf
The norm and the condition estimate of the transformation matrix
that takes the $\lambda_j(u)$ values into the $\mu_j(u)$ values,
based on \eref{eqmufromlambda},
are given in Figure \RSref{Fig020301500p3cond}, behaving roughly like $N^{0.69}$
and $N^{1.7}$, respectively. This is like
$h_\Omega^{-1/3}$ and $h_\Omega^{-2/3}$, respectively, in terms of the fill distance
$h_\Omega$ in
the domain.
\begin{figure}[hbt]
\begin{center}
\includegraphics[height=6.0cm,width=6.0cm]{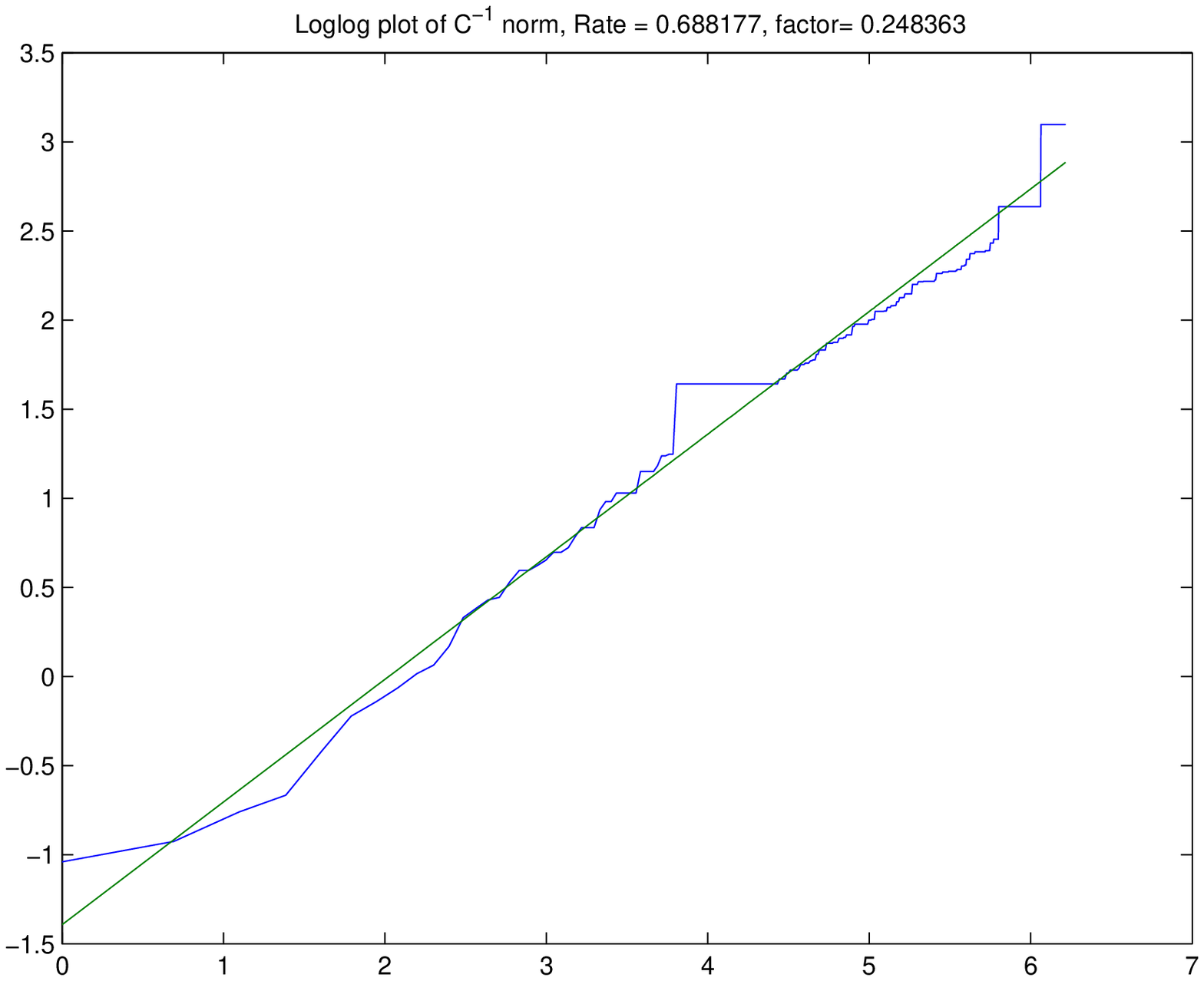}
\includegraphics[height=6.0cm,width=6.0cm]{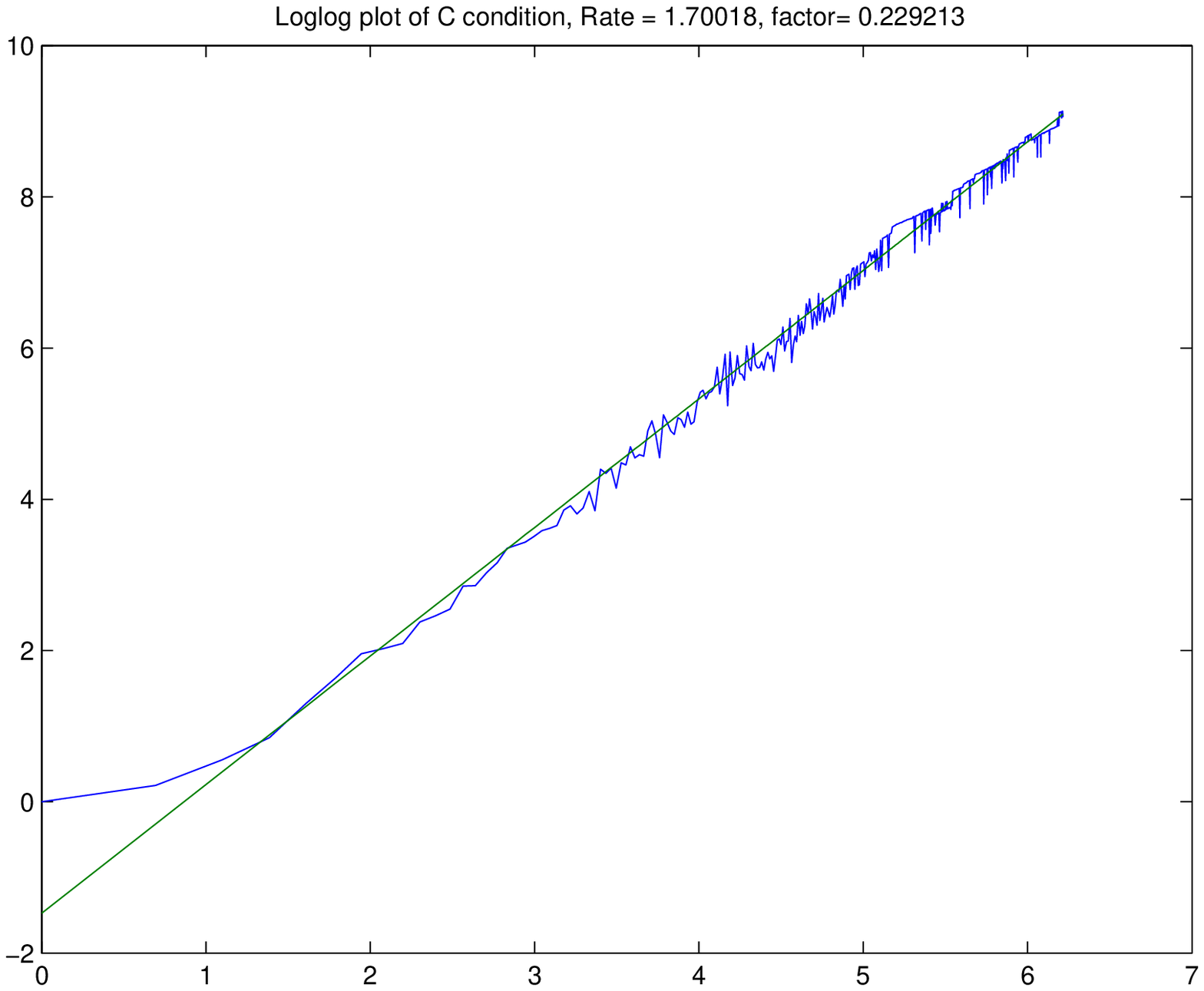} 
\end{center}
\caption{Norm and condition estimate of matrix $C(N)$ as functions of $N$.
\RSlabel{Fig020301500p3cond}}
\end{figure} 
\biglf
To get an idea of the decay of the basis functions $v_{\mu_j}$
in $L_2$ context,
the matrix of the values in a fine point set on the domain is calculated
and a singular value decomposition is done on that matrix.
The result is shown in Figure \RSref{Fig020301500p3sing}, the decay behavior 
of singular values being roughly like $N^{-2.4}$.
\biglf
Note that each basis function
is a worst case for the preceding steps, because it has zero data for them
and is approximated by the zero function. Thus $L_\infty$ or $L_2$ norms of the
basis functions are closely related to the worst-case
$L_\infty$ or $L_2$ norms of solutions with Hilbert space norm one.
Figure \RSref{Fig020301500p3RMSQSup} shows the RMSQ and
$L_\infty$ norms of the basis functions $v_{\mu_j}$ as functions of $j$,
the estimated decay being like $N^{-1.65}$ and $N^{-1.44}$, respectively,
but with serious roundoff pollution for large $j$. The peaks are exactly where
boundary points are chosen by the greedy method, and this is explained
via \eref{eqPKv} by the identity
$$
v_{\mu_j}^2(x)=P^2_{\Lambda_{j-1}}(\delta_x)-P^2_{\Lambda_{j}}(\delta_x),
\;P^2_\emptyset(\delta_x)=K(x,x)
$$
that calls for a large $v_{\mu_j}$ when there is a sharp drop
from $ \rho_{\Lambda_{j-1}}$ to $\rho_{\Lambda_{j}}$.
\begin{figure}[hbt]
\begin{center}
\includegraphics[height=8.0cm,width=8.0cm]{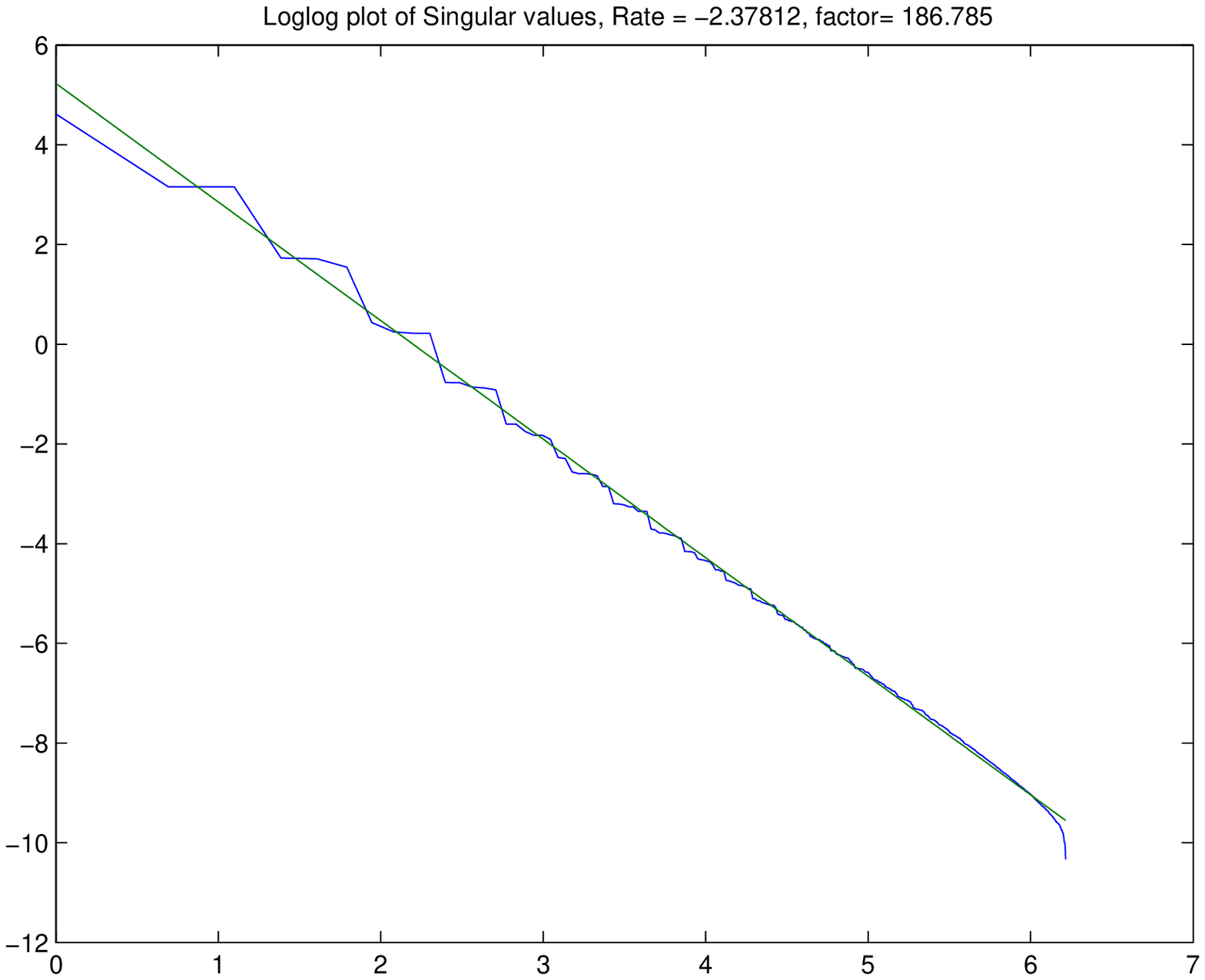}\\ 
\end{center}
\caption{Singular value decay of the matrix of suggested basis functions.
\RSlabel{Fig020301500p3sing}}
\end{figure} 
\begin{figure}[hbt]
\begin{center}
\includegraphics[height=6.0cm,width=6.0cm]{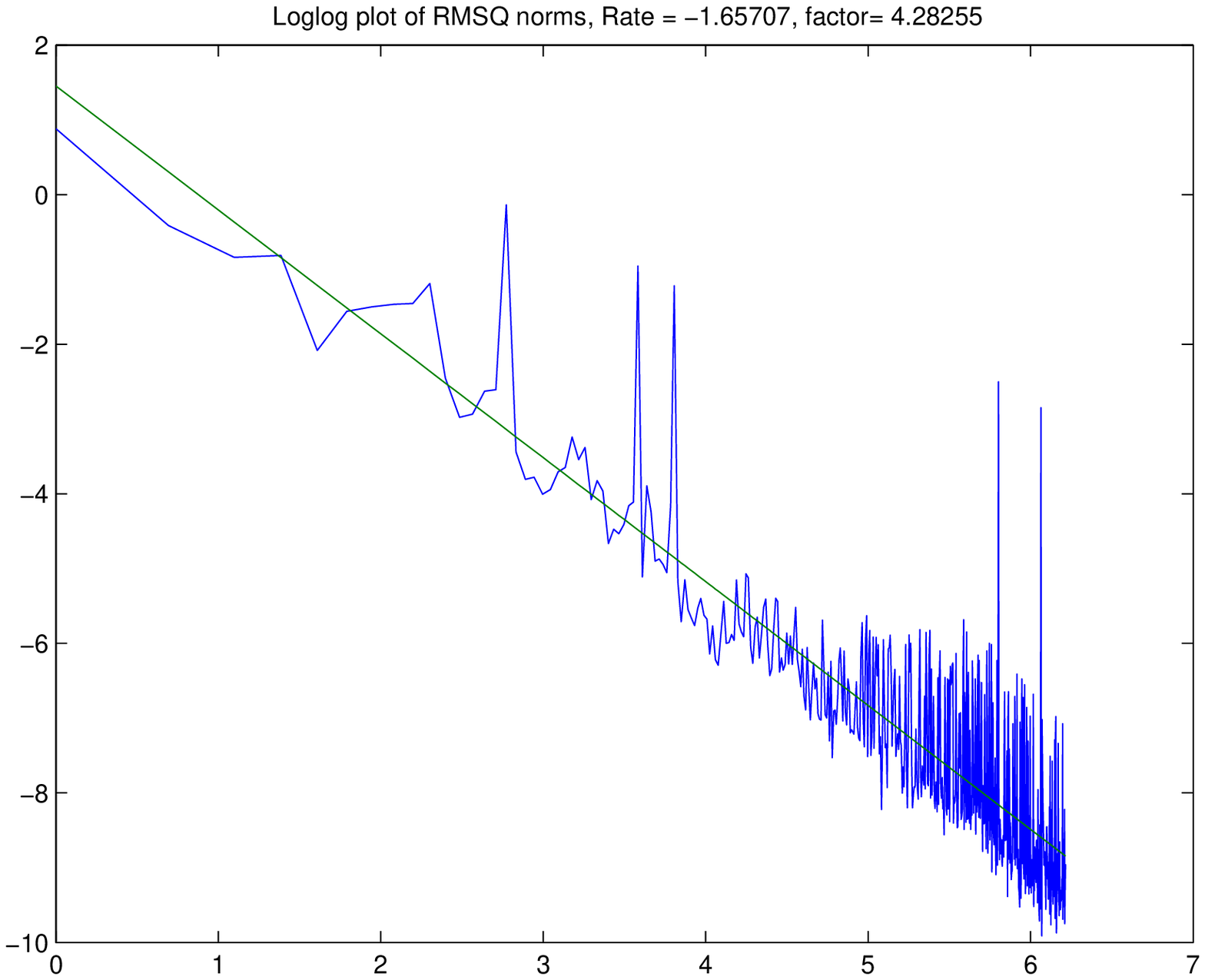}
\includegraphics[height=6.0cm,width=6.0cm]{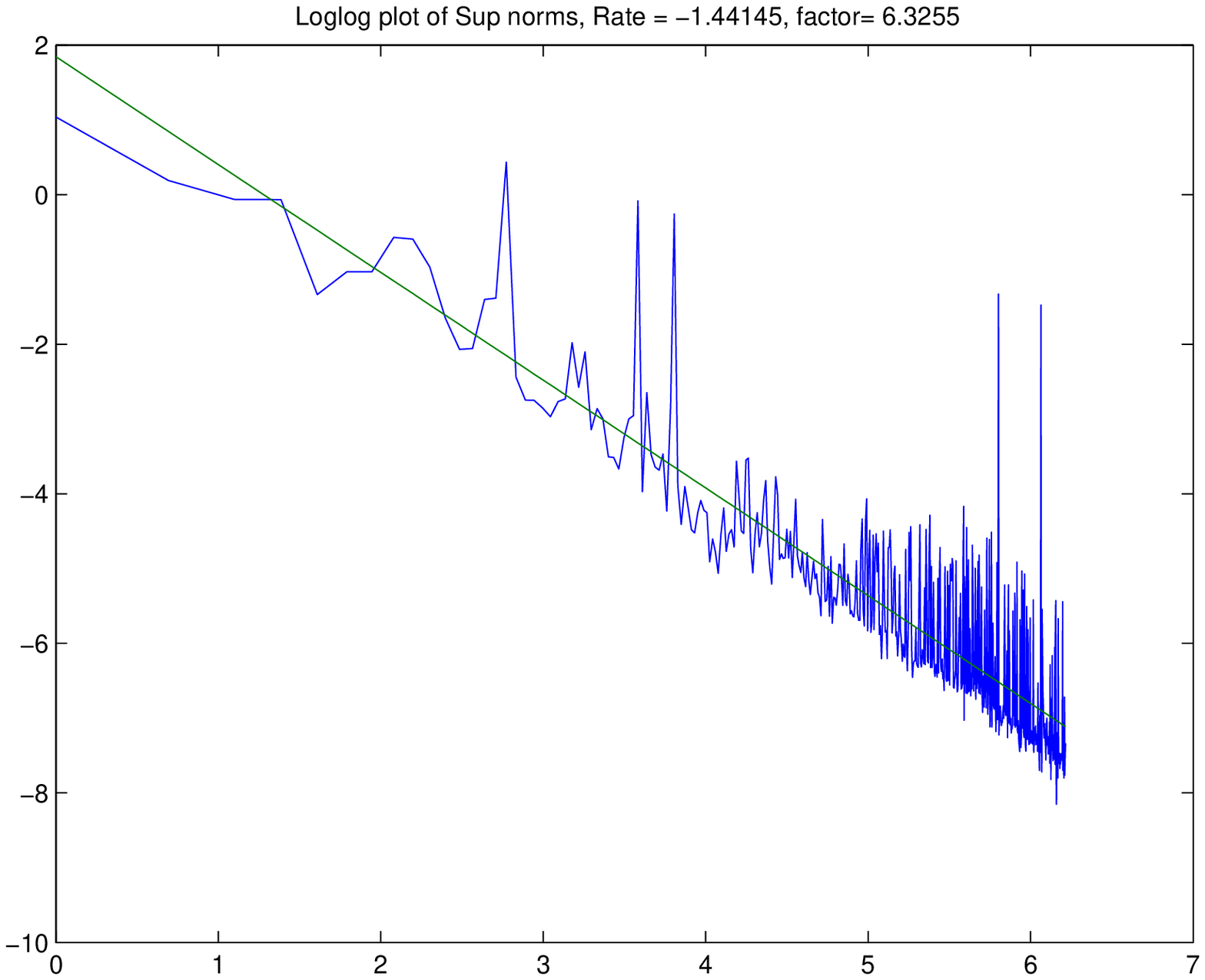} 
\end{center}
\caption{RMSQ and Sup 
\RSlabel{Fig020301500p3RMSQSup}}
\end{figure} 
\subsection{Observations for varying smoothness}\RSlabel{SecNumResObsVm}
We now check the behavior of the greedy method when the smoothness parameter $m$
of the Hilbert spaces $W_2^m(\R^d)$ changes. In view of
\eref{eqhDexp} and \eref{eqhBexp}, and since we saw before that
the greedy method focuses on the differential operator and the domain,
not on boundary values, Figure \RSref{Fig031500prunlamdelta}
shows rates $-0.67$ and $-0.71$ as functions of $m$ that
confirm the $-(m-3)/2$ rate of \eref{eqhDexp}
for both $\sigma_{\Lambda_{500}}$ and $\rho_{\Lambda_{500}}$ as functions of $m$
after 500 steps of the greedy method.
\begin{figure}[hbt]
\begin{center}
\includegraphics[height=6.0cm,width=6.0cm]{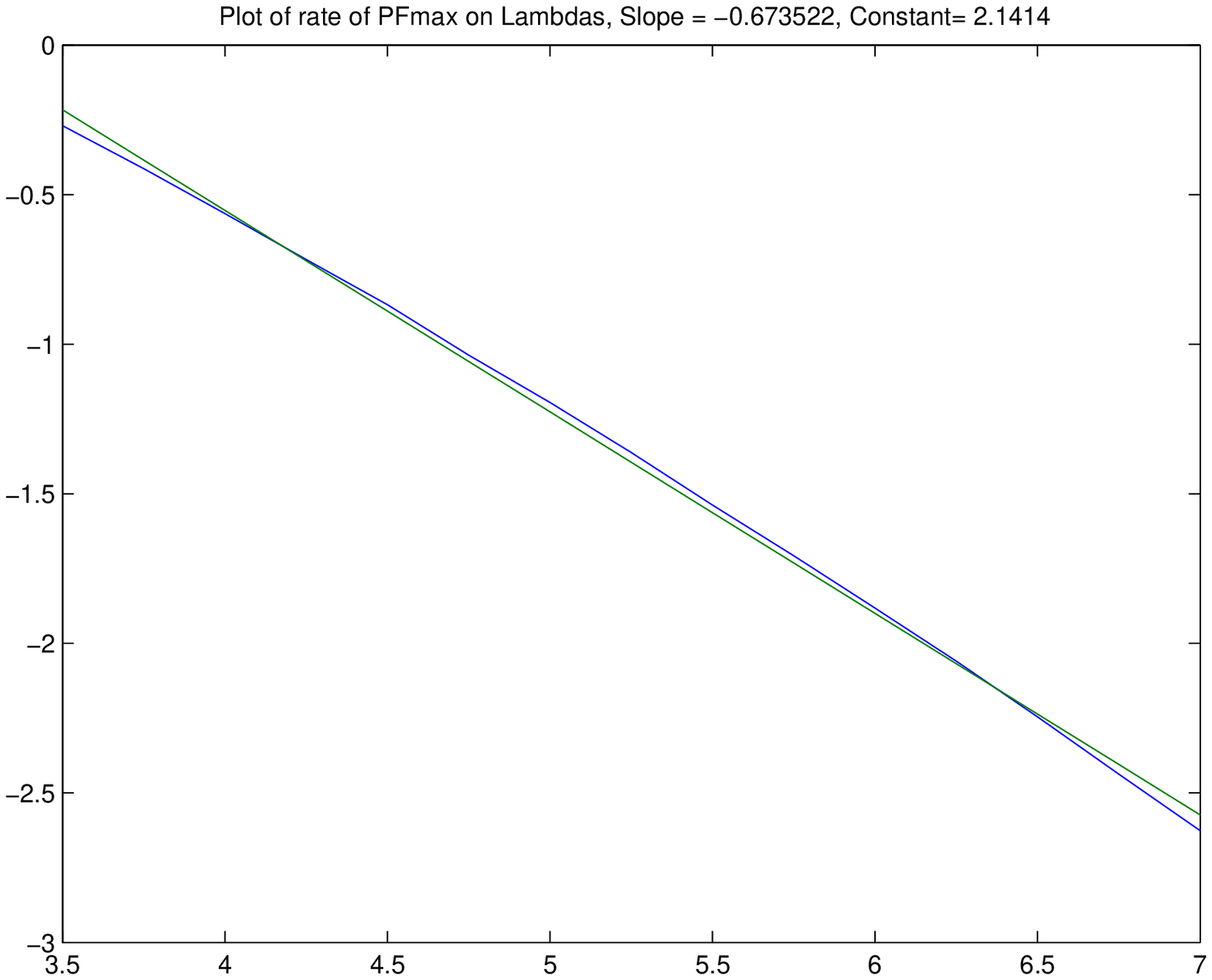}
\includegraphics[height=6.0cm,width=6.0cm]{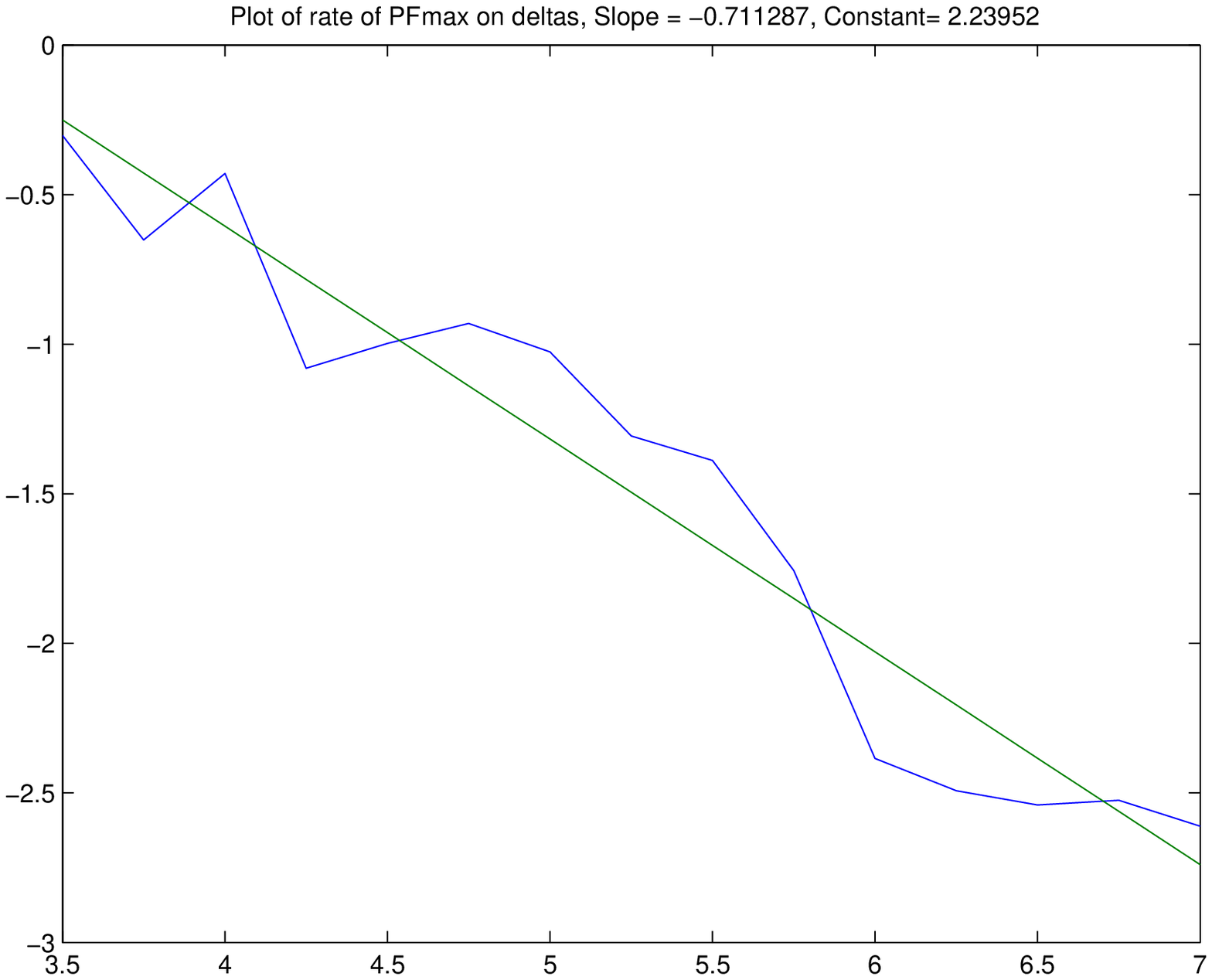}
\end{center}
\caption{$\sigma_{\Lambda_{500}}$ and $\rho_{\Lambda_{500}}$
  as functions of smoothness $m$, with rates
\RSlabel{Fig031500prunlamdelta}}
\end{figure} 
\subsection{Use of Basis Functions}\RSlabel{SecNumResUBF}
To get examples of solving specific Poisson problems via \eref{equNtilde} using the basis
obtained by the Greedy Method, we selected a run for smoothness $m=6$
to generate the basis first, the other parameters being as in the previous
examples. Then the functions $f$ and $g$ of
\eref{eqLufg} were defined to let the true solution be of the form
$K(\cdot,z)$ for a different kernel $K$ and a point $z=(-\pi/10, 0)$.
This allows to check cases with solutions of different smoothness.
\biglf
The first case is the infinitely smooth situation where $K$ is a Gaussian.
Figure \RSref{figg201} shows a very fast decay of the error, a fast increase of
the cumulative sum of the coefficients $\mu_j^2(u)$ from \eref{equNtilde},
and the final  maximal error $8\cdot 10^{-6}$, roughly.
A less smooth solution is of the form $r^{2.5}$ with $r=\|\cdot-z\|_2$
with a derivative singularity at $z$, and the corresponding results are in
Figure \RSref{figq201}. The $\mu_j^2$ decay much more slowly, and the
solver has to fight with the derivative singularity at $z$. The two cases
were aligned by factors to start with an error of roughly one
using the same precalculated basis. Recall that
classical Harmonic Analysis shows that decay rates
of coefficients of orthonormal expansions depend on smoothness and
determine convergence  rates. 
\begin{figure}[hbt]
\begin{center}
\includegraphics[height=4.0cm,width=4.0cm]{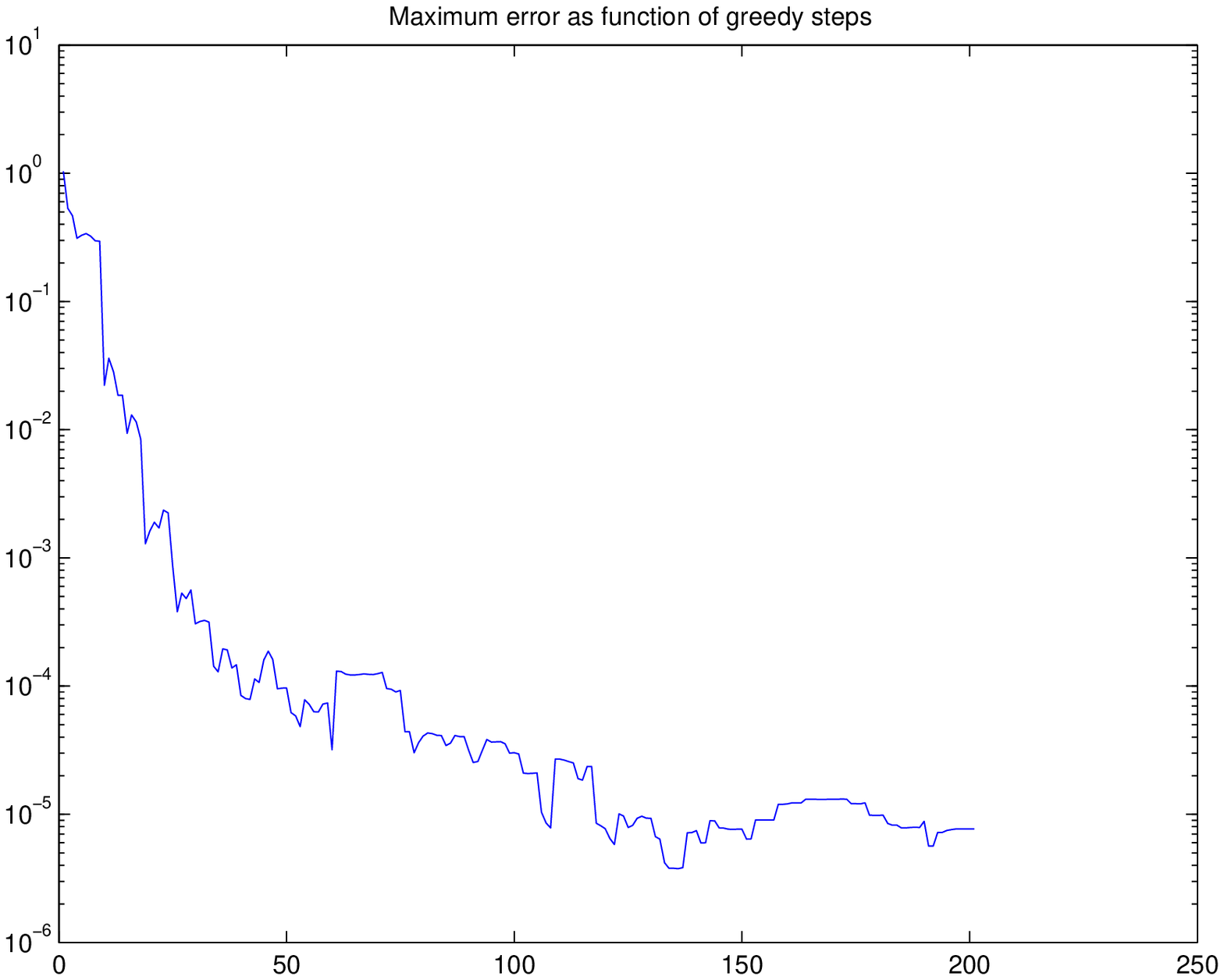} 
\includegraphics[height=4.0cm,width=4.0cm]{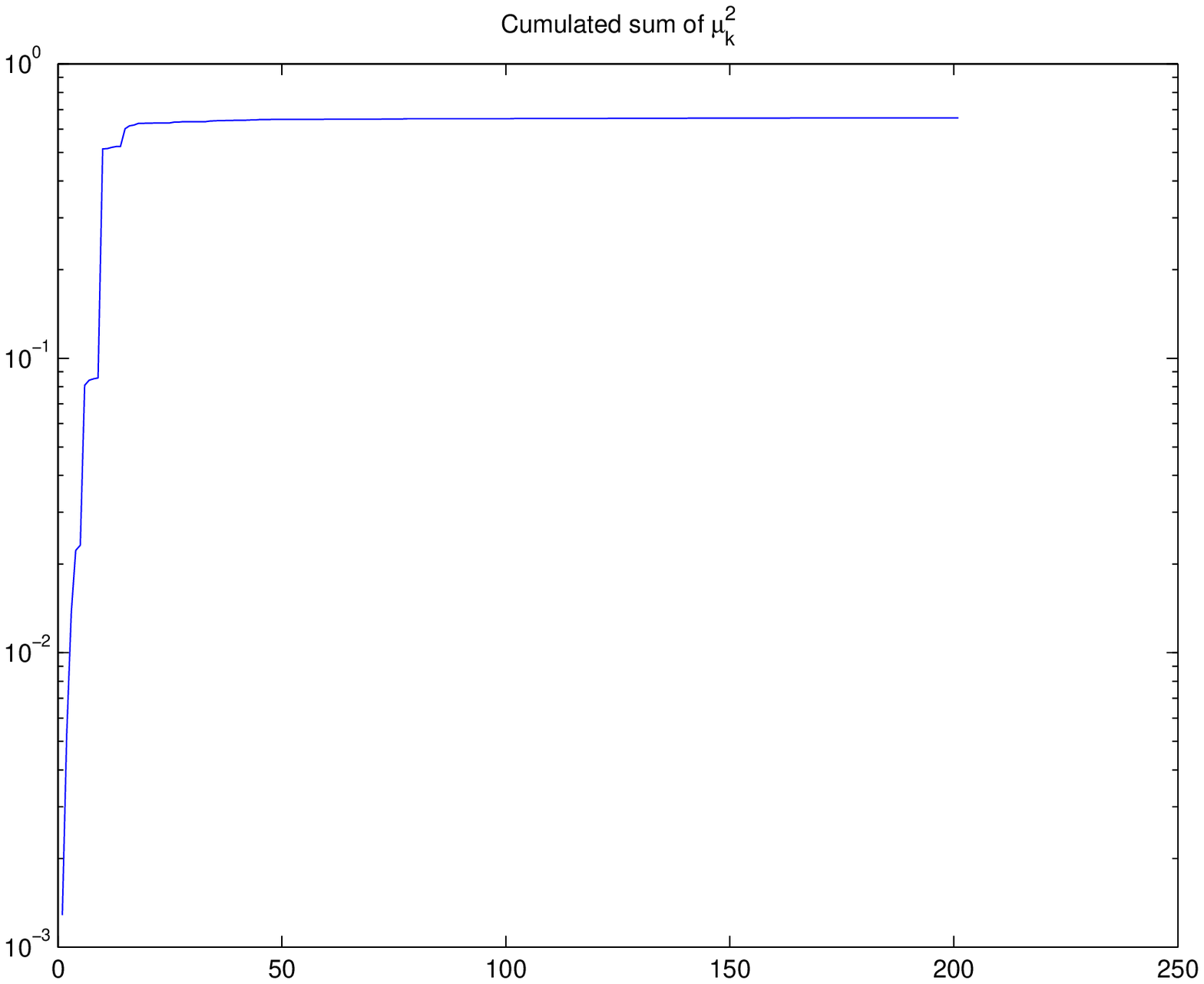} 
\includegraphics[height=4.0cm,width=4.0cm]{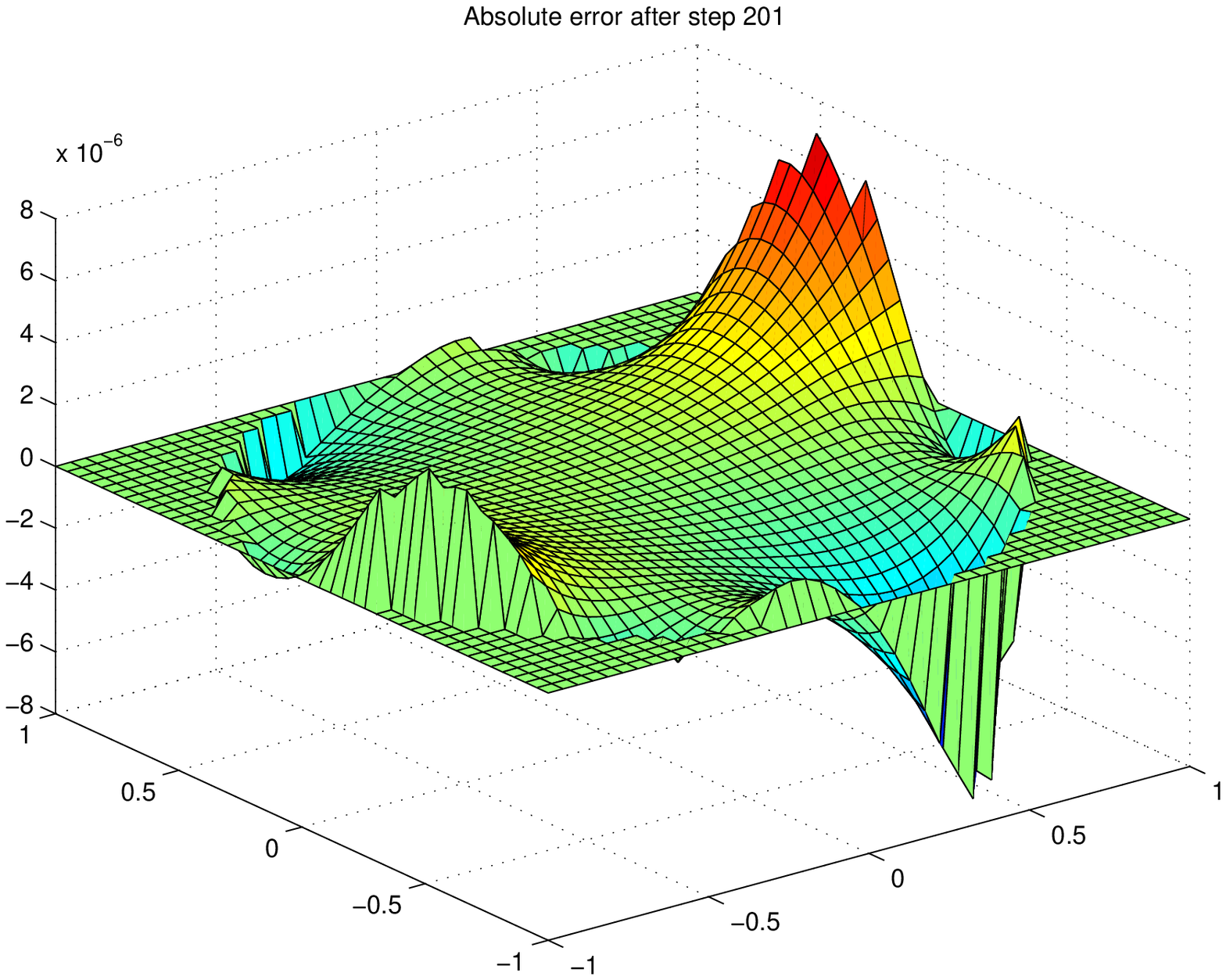} 
\end{center}  
\caption{Maximum error and cumulative sum of the $\mu_j^2(u)$ as functions of $N$, and
  final error about  $8\cdot 10^{-6}$, Gaussian case\RSlabel{figg201}} 
\end{figure}

\begin{figure}[hbt]
\begin{center}
\includegraphics[height=4.0cm,width=4.0cm]{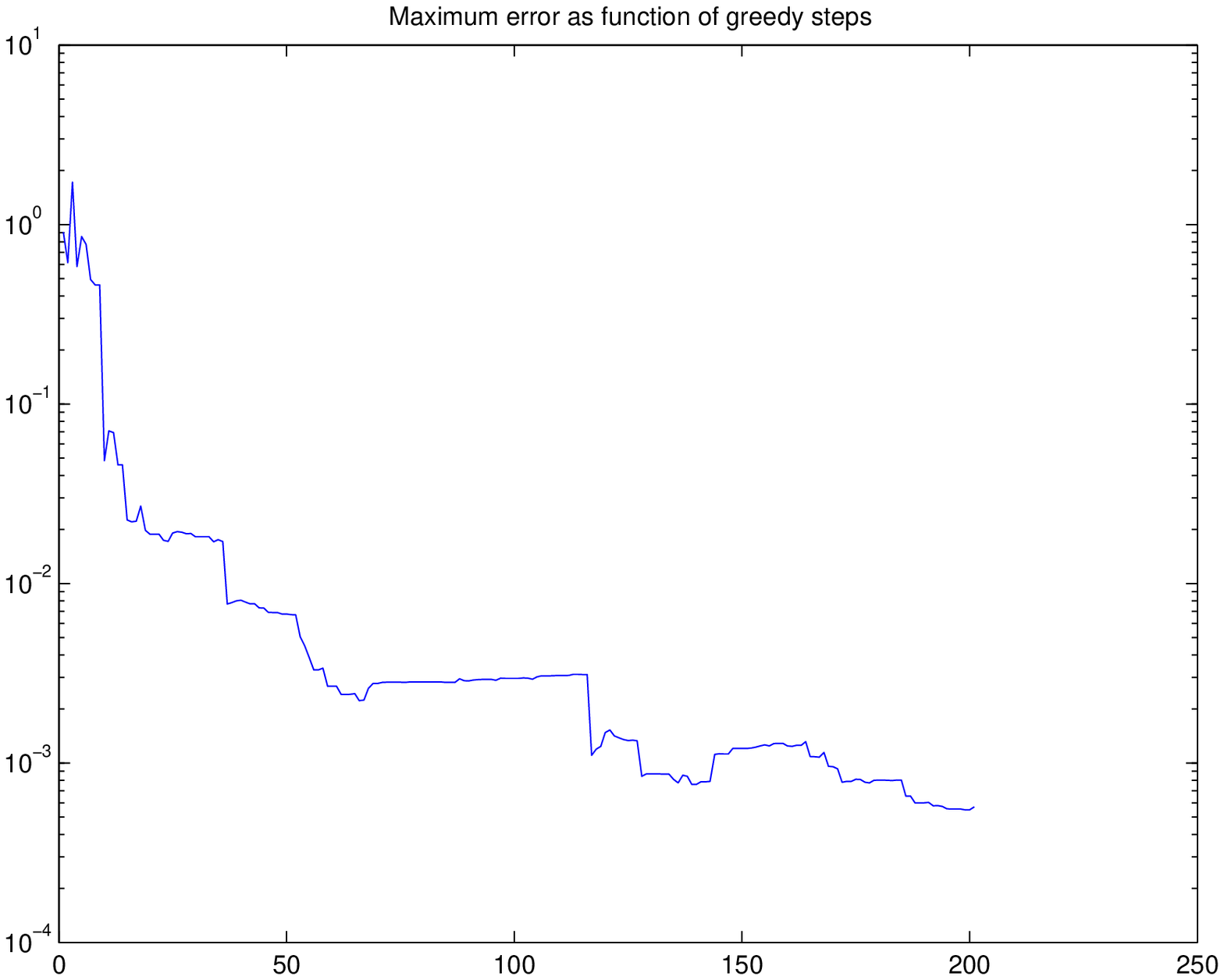} 
\includegraphics[height=4.0cm,width=4.0cm]{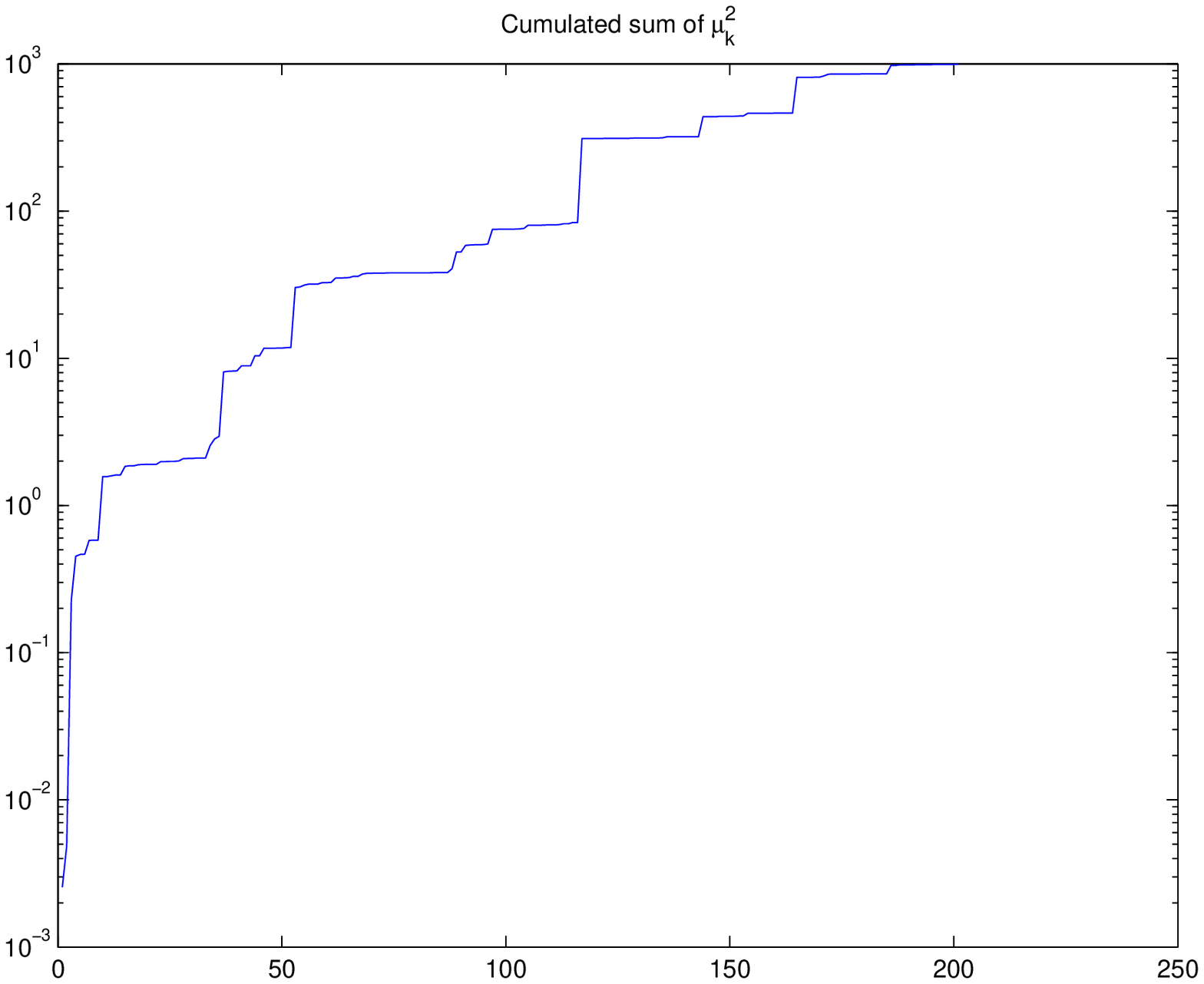} 
\includegraphics[height=4.0cm,width=4.0cm]{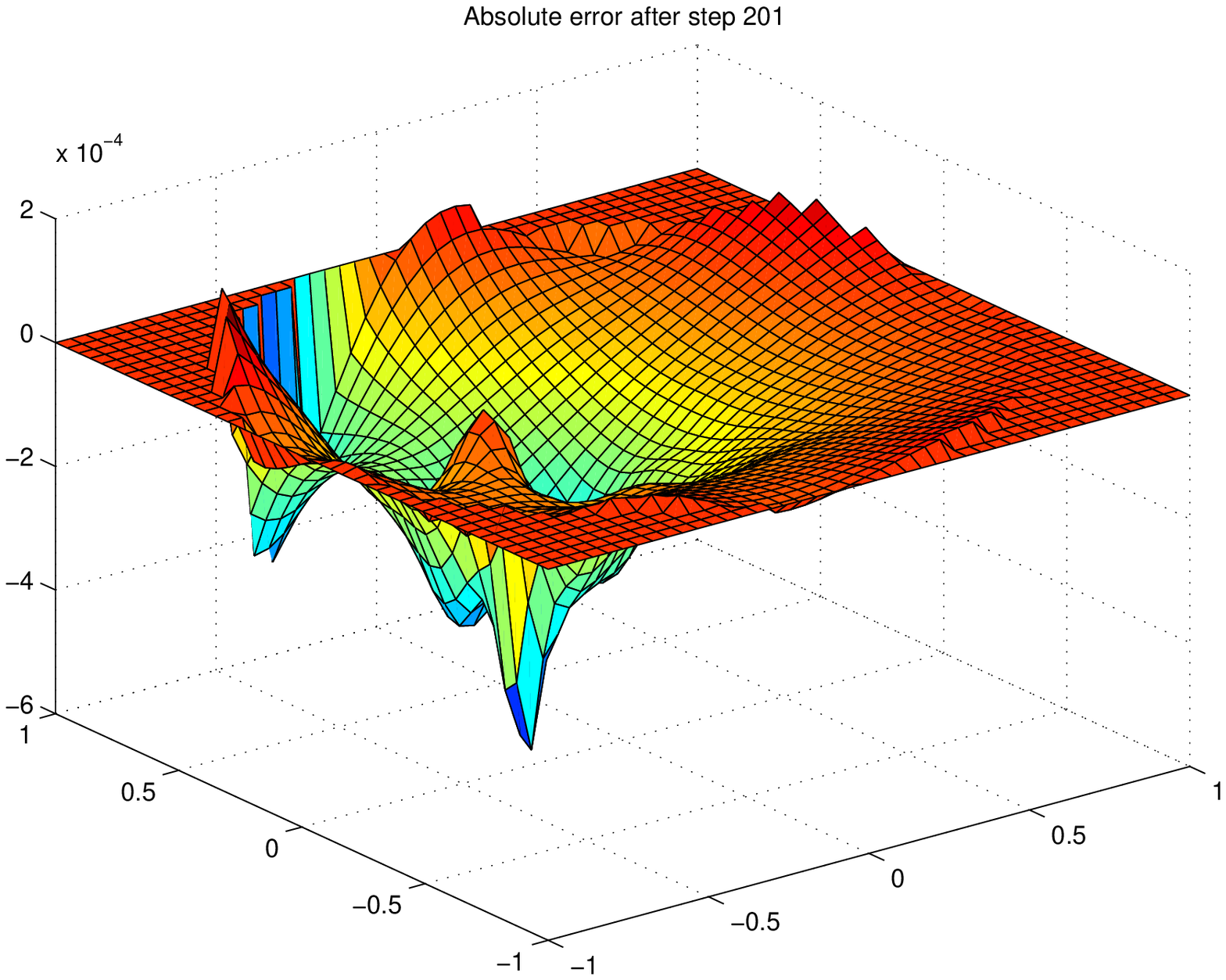} 
\end{center}  
\caption{Maximum error and cumulative sum of the $\mu_j^2(u)$ as functions of $N$, and
  final error about  $4\cdot 10^{-5}$, Power case $r^{2.5}$\RSlabel{figq201}} 
\end{figure}
\subsection{The Extended Greedy Method}\RSlabel{SecNumResExtGM}
The Extended Greedy Method from Section
\RSref{SecVGM} was run in the same situations as in Sections
\RSref{SecNumResObsCVN} and \RSref{SecNumResObsVm}.
The results are given
from Figure \RSref{Fig040301500p3rates} on. 
\begin{figure}[hbt]
\begin{center}
\includegraphics[height=4.0cm,width=6.0cm]{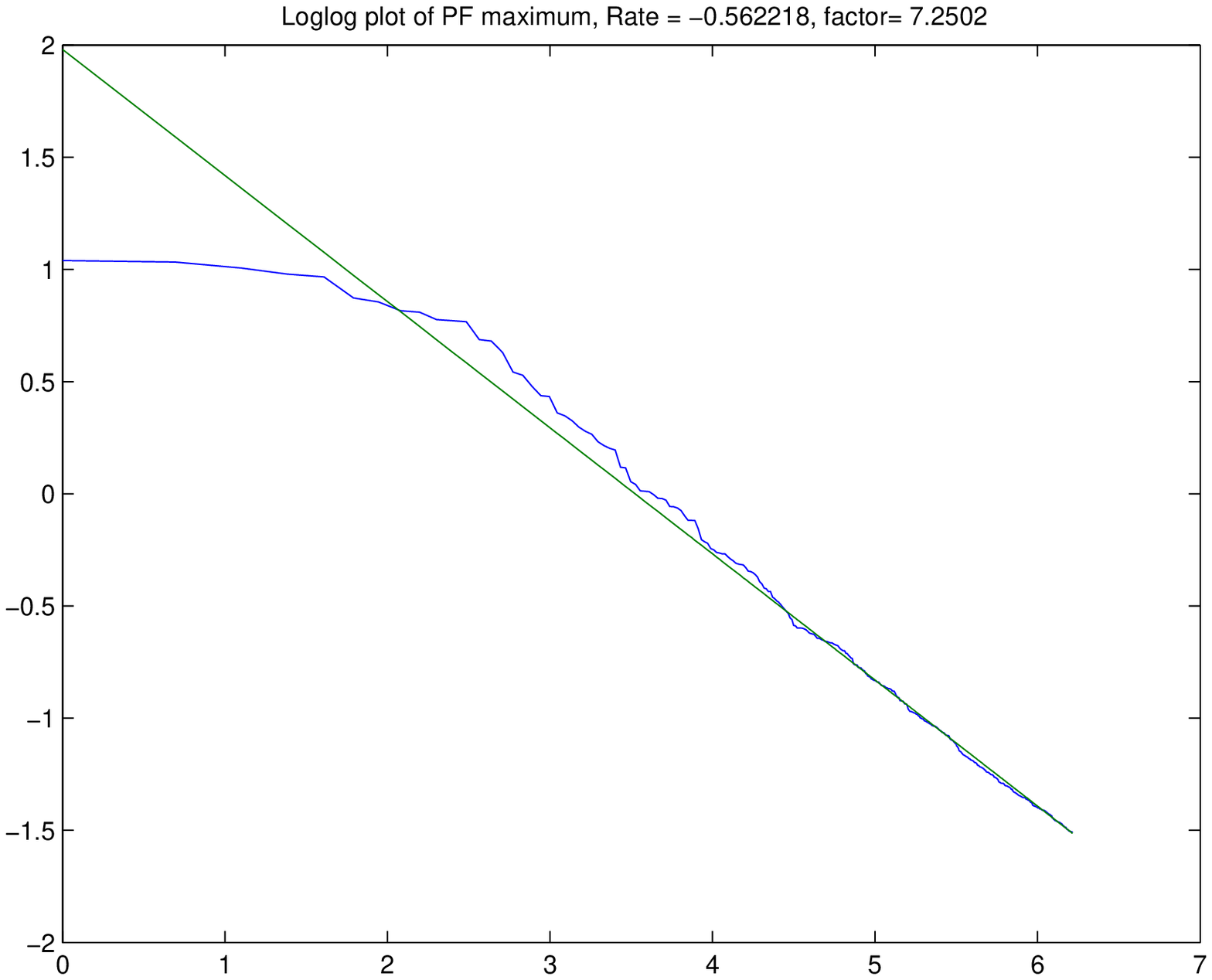} 
\includegraphics[height=4.0cm,width=6.0cm]{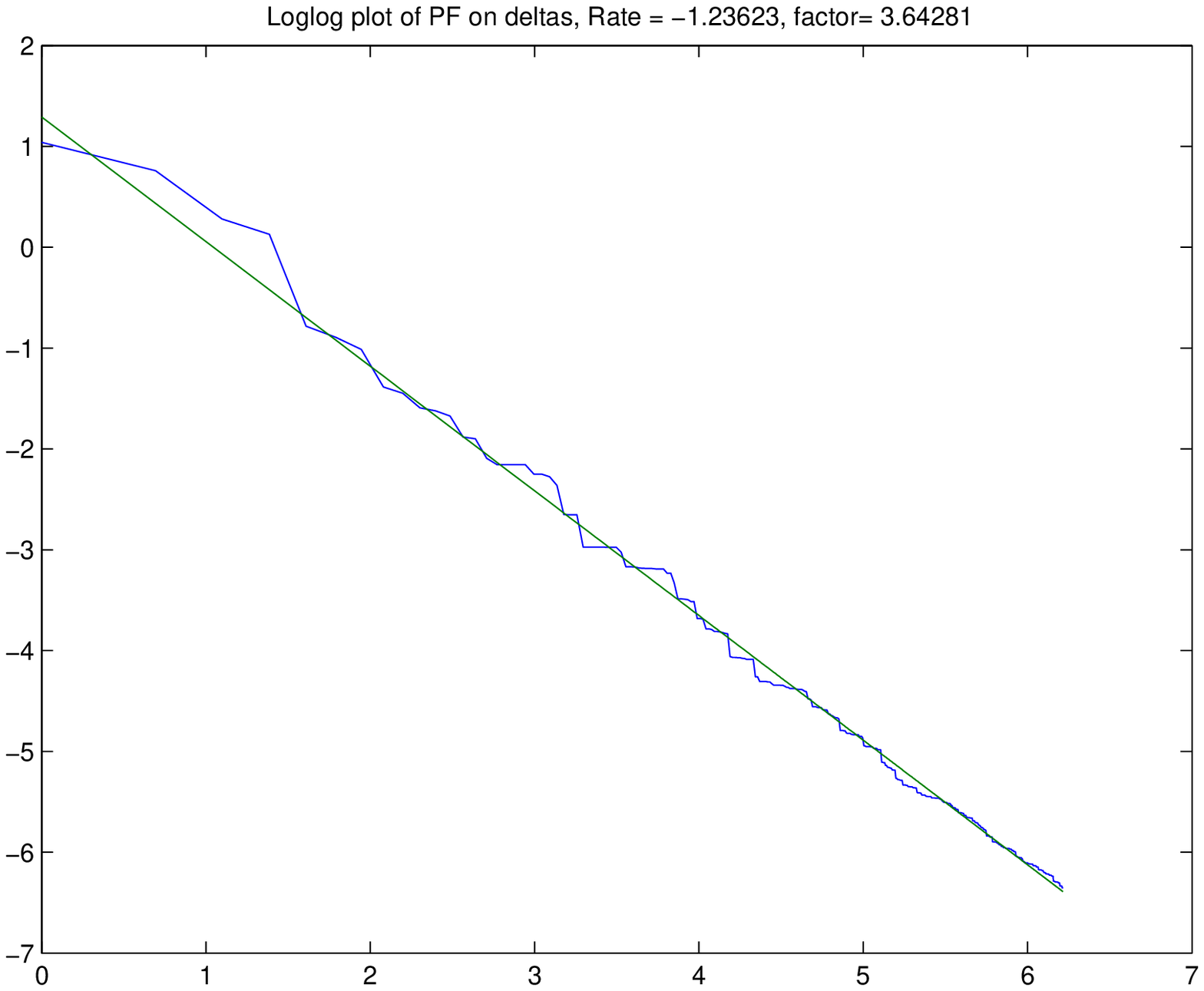} 
\end{center}
\caption{$\sigma_{\Lambda_N}$ and $\rho_{\Lambda_N}$
  and their rates as functions of $N$, for the extended Greedy Method
  \RSlabel{Fig040301500p3rates}}
\end{figure} 
\biglf
The rates in the left parts of
Figures \RSref{Fig020301500p3rates} and \RSref{Fig040301500p3rates} are similar,
but not in the right parts. 
Since the extended method picks more boundary points than the original method,
the convergence rate on the delta functionals is now much better. Similarly, 
Figures \RSref{Fig020301500p3both} and \RSref{Fig040301500p3both} differ
considerably. The error in the interior decays much better, and the
fill distances in the domain and on the boundary show a better alignment.  
\begin{figure}[hbt]
\begin{center}
\includegraphics[height=4.0cm,width=6.0cm]{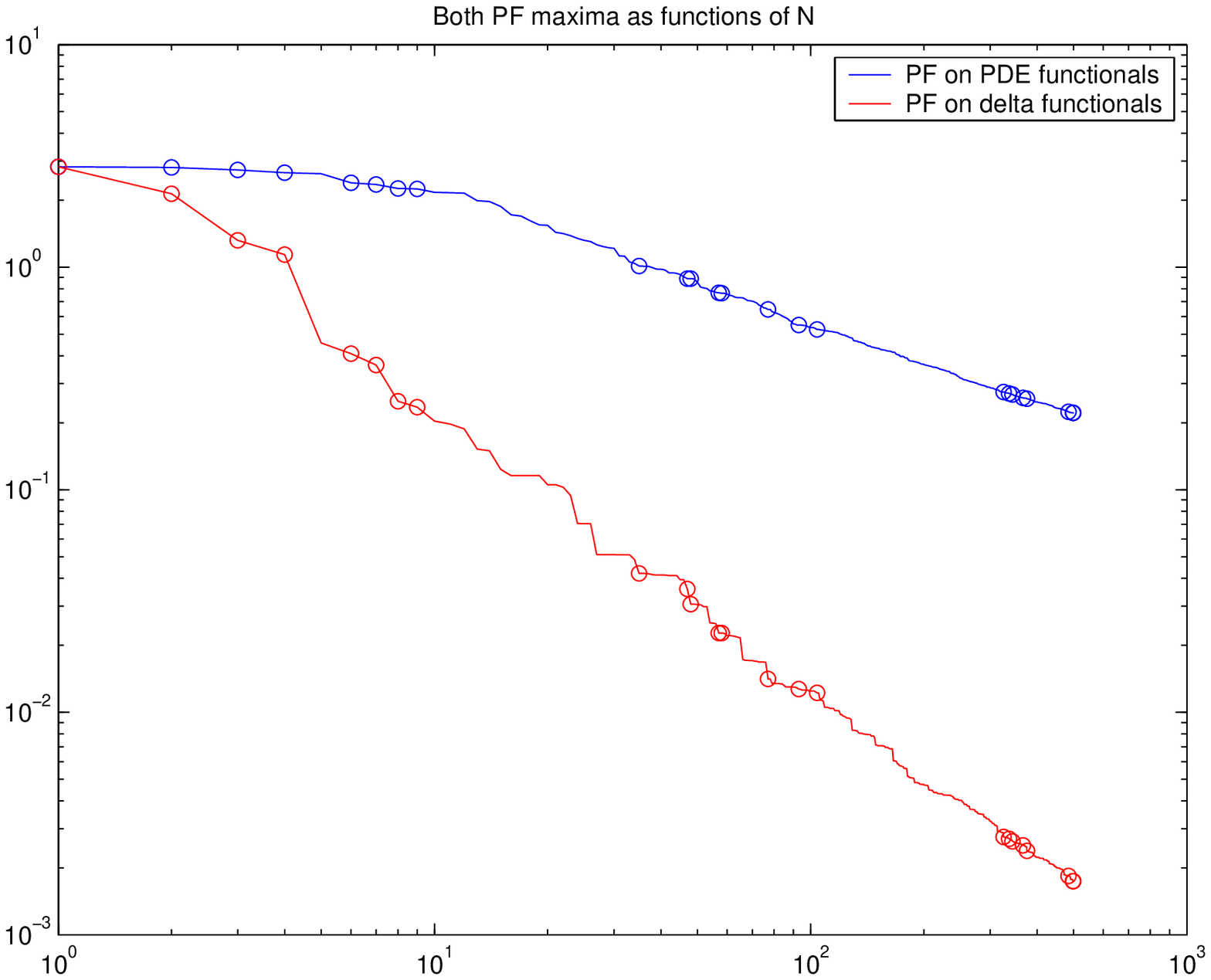}\\ 
\includegraphics[height=4.0cm,width=6.0cm]{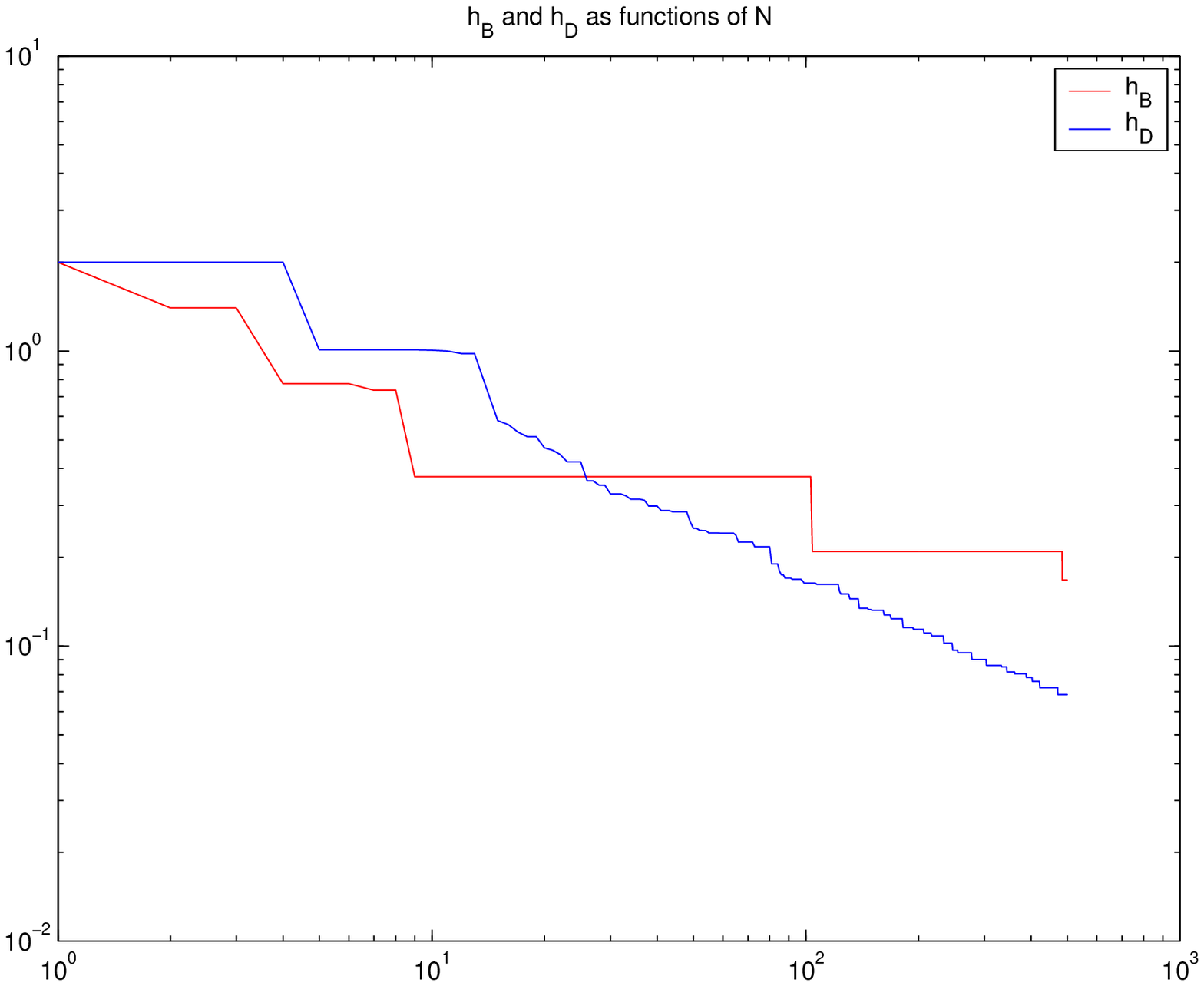} 
\end{center}
\caption{$\sigma_{\Lambda_N}$ and $\rho_{\Lambda_N}$ (top)
  compared with $h_\Gamma(N)$ and $h_\Omega(N)$ (bottom) 
  as functions of $N$, for the extended Greedy Method. \RSlabel{Fig040301500p3both}}
\end{figure}

\biglf
Figure \RSref{Fig040301500p3PF} adds a plot of the Power Function
on the domain to the three plots of Figure \RSref{Fig020301500p3PF}, and it
should be compared to Figure \RSref{Fig020301500p3NB}.
One can see that the improved selection of boundary functionals now avoids large
values of the Power Function on the boundary.
The other results are very similar, and plots are omitted, except for
the condition. The extended method is less stable, if 
Figures \RSref{Fig020301500p3cond} and
\RSref{Fig040301500p3cond} are compared. 

\begin{figure}[hbt]
\begin{center}
\includegraphics[height=8.0cm,width=12.0cm]{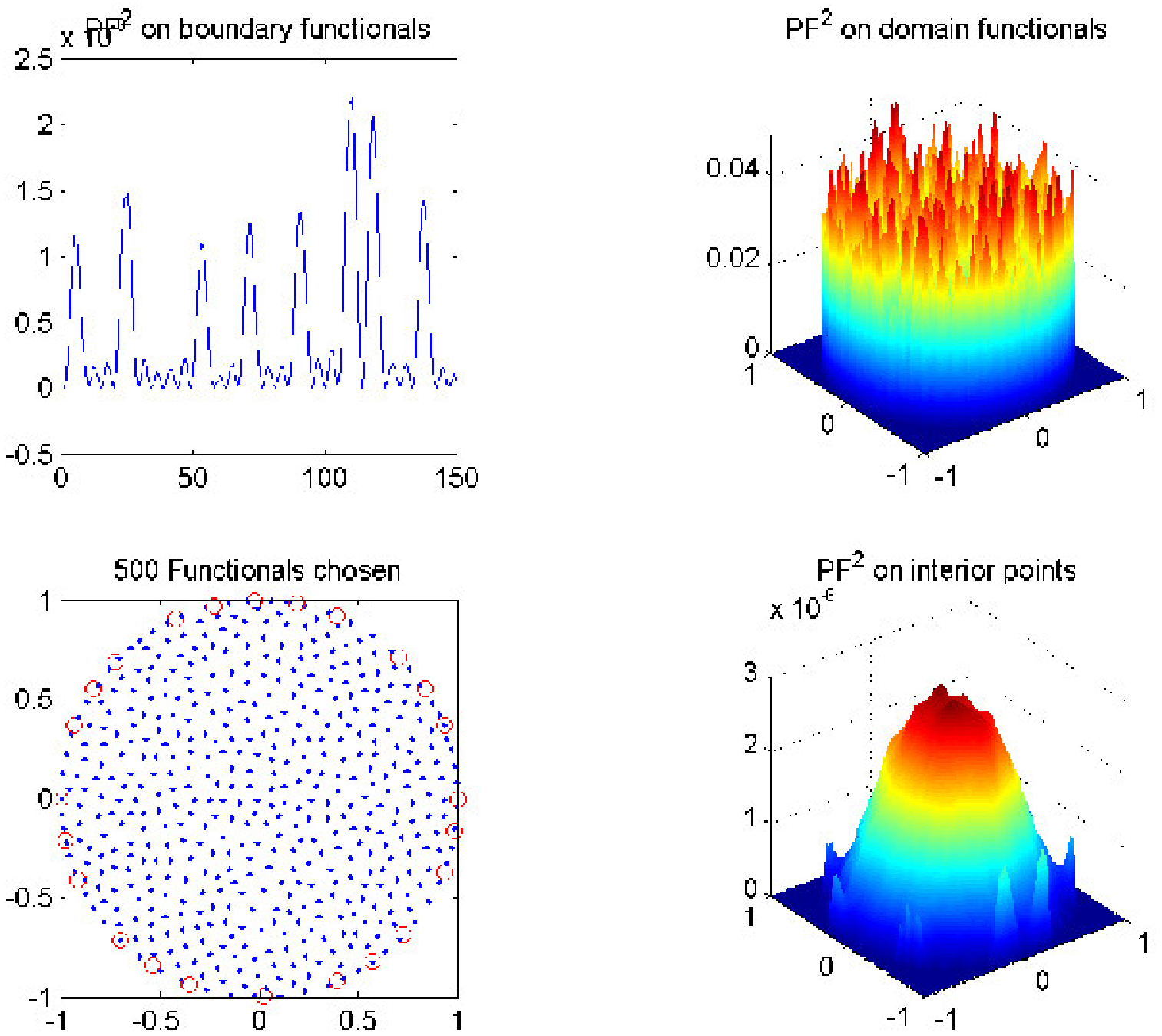}\\ 
\end{center}
\caption{$P^2_{\Lambda_{500}}$ on the boundary functionals, the domain
  functionals, and the selected 500 functionals, for the extended Greedy Method.
\RSlabel{Fig040301500p3PF}}
\end{figure} 
\begin{figure}[hbt]
\begin{center}
\includegraphics[height=6.0cm,width=6.0cm]{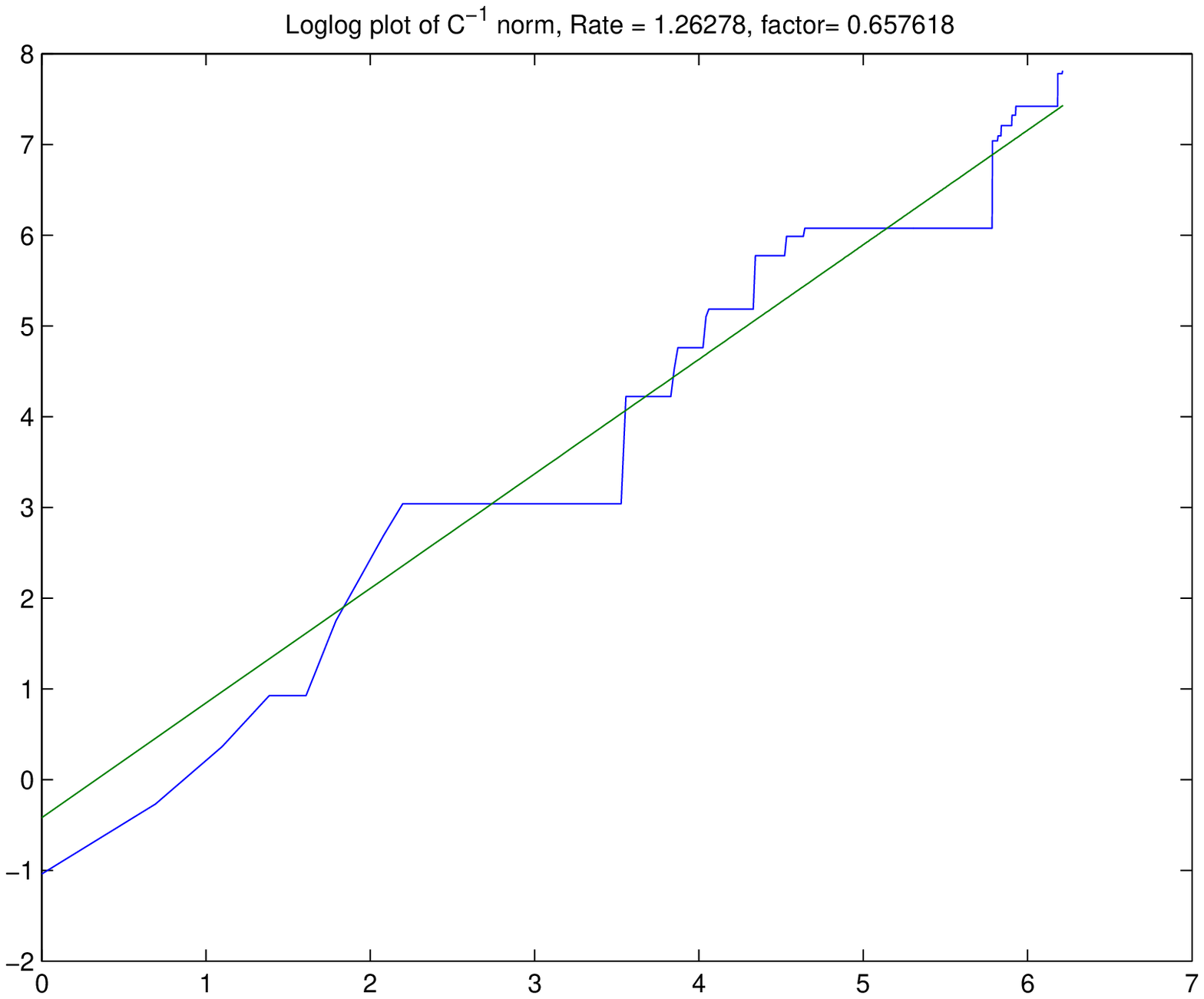}
\includegraphics[height=6.0cm,width=6.0cm]{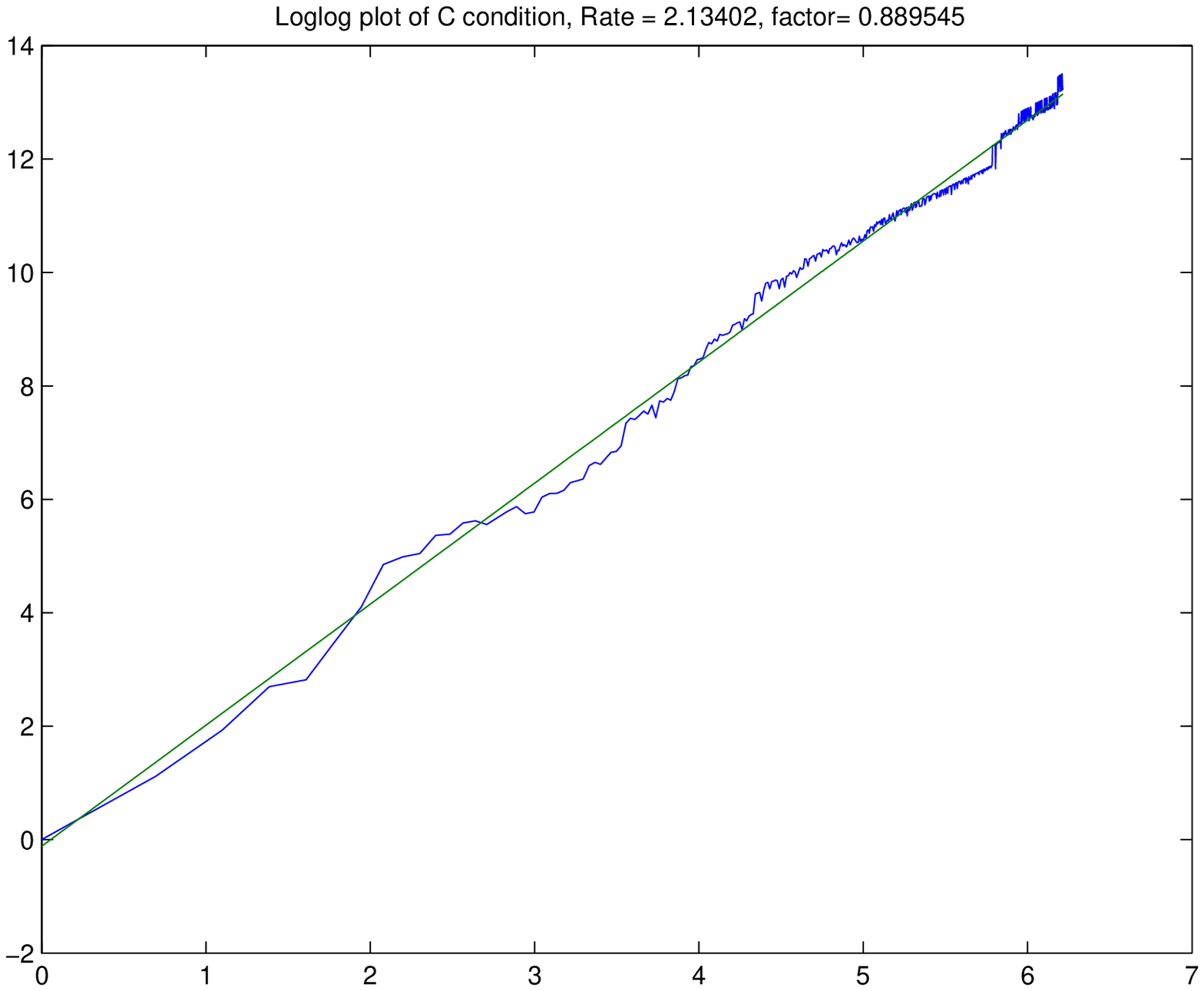} 
\end{center}
\caption{Norm and condition estimate of matrix $C(N)$ as functions of $N$,
  for the extended Greedy Method.
\RSlabel{Fig040301500p3cond}}
\end{figure} 
\biglf
From Lemma \RSref{Lemrhosig}
one might hope that the extended Greedy Method performs better as a function of
smoothness, but this cannot be supported by experiments.
The rates for the situation corresponding to Figure \RSref{Fig031500prunlamdelta} come out to
be roughly -0.69 and -0.71, respectively, but the right-hand plot is much
smoother. The actual plot is suppressed.
\section{Summary and Open Problems}\RSlabel{SecSaOP}
Roughly, the $P$-greedy method for solving Dirichlet problems for
second-order elliptic
operators in $W_2^m(\R^d)$ seems to behave like the comparable
$P$-greedy method
for interpolation of functions in $W_2^{m-2}(\R^d)$, with all its pros and cons.
It focuses on the domain, not on the boundary, and it tends to
produce an asymptotically uniform distribution of evaluation points there,
with an unexpectedly small number of points for sampling the boundary values.
\biglf
On the theoretical side, this opens the quest for a thorough analysis
of Kolmogoroff $N$-widths for such PDE problems. A reasonable hypothesis is
that these behave like those without differential operators, but
for spaces of functions with lower-order smoothness. 
\biglf
Another observation is that the
maximum of the generalized Power Function taken
on all delta functionals, being a central quantity for
pointwise error bounds, shows the same asymptotics as the
maximum of the generalized Power Function taken on all chosen
PDE data functionals. This is no surprise for
well-posed problems, but it opens a way for
explicitly computable factors for error bounds in terms of the
Hilbert space norm of the true solution.
\biglf
We supplied an Extended Greedy Method that cares for the delta
functionals in a better way, but it falls out of the Hilbert space foundation,
so far. Its analysis is another open problem.
\biglf
Using the standard Hilbert space background
\RScite{schaback:2009-8,hon-schaback:2013-1}, we also can formulate
a $P$-greedy method for solving Dirichlet problems for harmonic
functions in 2D or 3D. It will follow the Kolmogoroff $N$-width
theory for such cases, but the latter seems to be open.
\biglf
There are other greedy techniques on the market
(``$f$--greedy'' and ``$f/P$-greedy'') that apply to specific
problems of the form \eref{eqLufg}, not uniformly to the whole class.
But their convergence analysis is less far developed,
see \RScite{santin-haasdonk:2017-1}.
\subsection*{Acknowledgment} Special thanks go to Gabriele Santin for
several helpful remarks.

\bibliographystyle{plain}

\end{document}